\documentclass[12pt,a4paper]{amsart}
\usepackage{amsmath, amsthm, amscd, amsfonts,graphicx}
\usepackage[matrix,arrow]{xy}
 \newtheorem{proposition}{Proposition}[section]
\newtheorem{theorem}{Theorem}[section]
\newtheorem{lemma}{Lemma}[section]
\newtheorem{corollary}{Corollary}[section]

\newtheorem{example}{Example}[section]
\newtheorem{examples}[example]{Examples}
\theoremstyle{definition}
\newtheorem{definition}{Definition}[section]

\newtheorem{remark}{Remark}[section]

\begin{document}

\title[Epireflections for models of sketches]{Very-well-behaved epireflections for categories of models of sketches}

 \author[Jo\~{a}o J. Xarez]{Jo\~{a}o J. Xarez}



 \address{CIDMA - Center for Research and Development in Mathematics and Applications,
Department of Mathematics, University of Aveiro, 3810-193 Aveiro, Portugal}

 \email{xarez@ua.pt}

 \subjclass[2020]{18A40,18A32,18E50,18N50,18F20,18C30,18N10}

 \keywords{Epireflection, stable units, prefactorization, monotone-light factorization, presheaves, simplicial set, Kan extensions, pseudo-filtered, cogenerating set, models of sketches, $n$-categories}

 \dedicatory{This article is dedicated to Robert Par\'{e} for his 80th birthday}

\begin{abstract}
Firstly, precise conditions on how to obtain \emph{very-well-behaved} epireflections are explored and improved from the author's previous papers; meaning that, beginning with a monad and a prefactorization system on a category, is produced a reflection with stable units (stronger than semi-left-exactness, also called admissibility in categorical Galois Theory) and an associated monotone-light factorization. Then, we were able to show that, for a pseudo-filtered category $\mathbb{J}$ in which every arrow is a monomorphism, the colimit functor on $Set^\mathbb{J}$ produces a \emph{very-well-behaved} epireflection; if $\mathbb{J}=\mathbf{2}$ the monotone-light factorization is non-trivial, as showed as an example. Then, new results are presented that grant \emph{very-well-behaved} subreflections from the \emph{\emph{very-well-behaved}} reflections induced by an adjunction given by right Kan extensions for presheaves. These subreflections are obtained by restricting to the models of a sketch; it is showed finally that the known \emph{very-well-behaved} reflection of n-categories into n-preorders is an example of this process (being n any positive integer).
\end{abstract}

\maketitle

\section{Introduction}\label{sec-introduction}

Let $(\mathcal{F},\theta)$ be a pointed endofunctor on a category $\mathbb{C}$ with pullbacks, i.e., $\mathcal{F}:\mathbb{C}\rightarrow \mathbb{C}$ is an endofunctor and $\theta:1_\mathbb{C}\rightarrow \mathcal{F}$ a natural transformation from the identity functor into $\mathcal{F}$ (see \cite[\S 1.1]{for:ins:sep:factorization}). If $(\mathcal{E},\mathcal{M})$ is an (orthogonal) factorization system on $\mathbb{C}$ such that all the morphisms in $\mathcal{E}$ are epimorphisms and $\mathcal{E}=\mathcal{E}'$, the largest subclass of $\mathcal{E}$ stable under pullbacks, then there exists a \emph{well-behaved} reflection $I:\mathbb{C}\rightarrow \mathbb{M}$ whose unit $\eta$ is given by the componentwise decomposition of $\theta=\mu\circ\eta$ using the $(\mathcal{E},\mathcal{M})$-factorization. Well-behaved means that the induced reflection $I$ has stable units, i.e., every unit morphism $\eta_C$ belongs to $\mathcal{E}'_I$, the largest subclass of $\mathcal{E}_I$ which is stable under pullbacks, being $(\mathcal{E}_I,\mathcal{M}_I)$ the reflective factorization system associated to the reflection $I$ (see \cite{CJKP:stab}). Having stable units is a stronger condition than semi-left-exactness (see \cite{CHK:fact}), also called admissibility in categorical Galois theory.

The process described in the previous paragraph allows that, from an adjunction from a category with pullbacks and a stable factorization system, it is induced a well-behaved reflection and the corresponding Galois theories.

It may happen that the induced well-behaved reflection is moreover \emph{very-well-behaved}, meaning that there is a monotone-light factorization system $(\mathcal{E}'_I,\mathcal{M}^*_I)$ associated to it (see \cite{CJKP:stab}). In order to be so, it only needs that, for every $C\in\mathbb{C}$, there is an effective descent morphism (see \cite{JST:edm}) into $C$ such that its domain is in the reflective subcategory.

In section \ref{sec-General Results} below it is presented a detailed study of the process outlined in the paragraphs above, followed by some generic examples and some results. These results were applied afterward to left and right Kan extensions for presheaves in sections \ref{sec:left Kan extensions} and \ref{sec:right Kan extensions} respectively.

In section \ref{sec:left Kan extensions}, a new class of monotone-light factorizations were presented, deriving from adjunctions given by left Kan extensions for presheaves. A non-trivial example of such was displayed in an astonishingly simple case (cf.\ \ref{example:colimit} below).

Then, in the rest of the paper, we focused in the more interesting context of right Kan extensions for presheaves, where there is an elegant characterization of the functors $K$ in $Set^K\dashv Ran_K$ for which the existence of a very-well-behaved induced reflection is assured (see \ref{theorem:RightKanCogeneratingVeryWellBehaved}).

The main new result of this paper (the crucial Lemma \ref{lemma:catModClosedUnderFactFinLim}) is displayed in the last section. Which allows to move to the new level of well-behaved subreflections derived from the induced well-behaved reflections, from an adjunction of right Kan extensions for presheaves. The subreflection is obtained considering just the models of a sketch (cf.\ the beginning of section \ref{sec:sub-reflectionsfromodels}). The last part of the paper states that this process, ending in very-well-behaved subreflections of models of presheaves, encompasses the already known very-well-behaved reflections from n-categories into n-preorders (for every positive integer n; cf.\ Examples \ref{example:m-l fact catsviapreord}, and subsections \ref{subsec:2-categories into 2-preorders}, \ref{subsec:n-categories into n-preorders}).

It is noticeable that our main examples are about presheaves, $\mathbb{S}=Set$. In future work, we hope to enlarge this setting with new examples without this restriction, namely with $\mathbb{S}$ a regular or a (Barr) exact category, and to explore the algebraic setting (it is well known that a variety of universal algebras can be identified with a category of models of a sketch for presheaves).

Anyway, we believe that in this work a very detailed and precise theoretical framework was established, which makes us optimistic about future developments applying it.\\

\textbf{Notice that all the basic concepts needed to read the following paper can be found in the bibliography, mostly in two items of it, namely \cite{SM:cat}  and \cite{CJKP:stab}.}

\section{Revision of General Results with Improvements}\label{sec-General Results}

\subsection{Derivable well-behaved epireflections}\label{subsec:derivable well-behaved epireflections}

Let us first show that, a \emph{well-behaved} epireflection is derivable from a monad in a category with pullbacks preserved by the endofunctor, provided there is a stable factorization system whose left class contains only epimorphisms (the needed conditions are in fact much weaker, as the reader shall discover; see Theorem \ref{theorem:stable units idempotent monad} below).

\begin{lemma}\label{lemma:pointed endofunctor}\

Let $(\mathcal{F},\theta)$ be a pointed endofunctor on $\mathbb{C}$ (pictured by \begin{picture}(55,10)(0,0)

\put (0,0){$\mathbb{C}$}\put (50,0){$\mathbb{C}$}

\put(10,12){\vector(1,0){35}}\put(25,15){$1_\mathbb{C}$}
\put(10,-6){\vector(1,0){35}}\put(25,-17){$\mathcal{F}$}

\put(23,0){$\Downarrow$}\put(32,0){$\theta$}

\end{picture} ).\vspace{20pt}

If there is a prefactorization system $(\mathcal{E},\mathcal{M})$ on $\mathbb{C}$ such that, for every $C\in\mathbb{C}$, each component factorizes as $$\theta_C=\mu_C\circ\eta_C:C\rightarrow\mathcal{I}(C)\rightarrow\mathcal{F}(C),$$ with $\eta_C\in\mathcal{E}$ and $\mu_C\in\mathcal{M}$, then there is a unique pointed endofunctor $(\mathcal{I},\eta)$ on $\mathbb{C}$ such that $\theta=\mu\cdot\eta$, where $\eta :1_\mathbb{C}\rightarrow \mathcal{I}$ and $\mu :\mathcal{I}\rightarrow \mathcal{F}$ are the obvious natural transformations.

\end{lemma}

\begin{proof}
The proof follows trivially from the definition of a prefactorization system (cf.\ \cite[\S 2.1]{CJKP:stab}). Confer the diagram

\begin{picture}(370,60)

\put(10,40){$C$}\put(25,43){\vector(1,0){60}}
\put(52,49){$\eta_C$}
\put(95,40){$\mathcal{I}(C)$}\put(130,43){\vector(1,0){60}}
\put(152,49){$\mu_C$}\put(200,40){$\mathcal{F}(C)$}

\put(15,33){\vector(0,-1){20}}\put(0,20){$f$}
\put(85,20){$\mathcal{I}f$}\put(105,33){\vector(0,-1){20}}
\put(195,20){$\mathcal{F}f$}\put(215,33){\vector(0,-1){20}}

\put(10,0){$C'$}\put(25,3){\vector(1,0){60}}
\put(52,9){$\eta_{C'}$}
\put(95,0){$\mathcal{I}(C')$}\put(130,3){\vector(1,0){60}}
\put(152,9){$\mu_{C'}$}\put(200,0){$\mathcal{F}(C')$,}

\end{picture}\vspace{5pt}

\noindent where $\mathcal{I}f$ is the unique morphism making it commutative.\end{proof}

\begin{lemma}\label{lemma:well-pointed endofunctor}

Under the conditions of Lemma \ref{lemma:pointed endofunctor}, if $\eta_C:C\rightarrow \mathcal{I}(C)$ is an epimorphism in $\mathbb{C}$ for every $C\in\mathbb{C}$, then the pointed endofunctor $(\mathcal{I},\eta)$ on $\mathbb{C}$ is well-pointed, i.e., $\mathcal{I}\eta =\eta \mathcal{I}$.

\end{lemma}

\begin{proof}
$\eta$ is an epimorphism in the functor category $\mathbb{C}^\mathbb{C}$ and

\hspace{100pt} $(\mathcal{I}\eta)\cdot\eta=\eta\circ\eta=(\eta\mathcal{I})\cdot\eta$\end{proof}

In the following Proposition \ref{proposition:idempotent endofunctor}, a statement in \cite[\S 1.1]{for:ins:sep:factorization} is proved in detail.

\begin{proposition}\label{proposition:idempotent endofunctor}

A well-pointed endofunctor $(\mathcal{I},\eta)$ on a category $\mathbb{C}$ is idempotent (i.e., $\eta\mathcal{I}$ is an isomorphism) if and only if $Fix(\mathcal{I},\eta)$ is a reflective full replete subcategory of $\mathbb{C}$ with reflection $\eta$, where $Fix(\mathcal{I},\eta)$ is determined by all the objects $M$($\in\mathbb{C}$) for which $\eta_M:M\rightarrow \mathcal{I}(M)$ is an isomorphism.

\end{proposition}

\begin{proof}
$Fix(\mathcal{I},\eta)$ is a replete subcategory of $\mathbb{C}$, since if $k:M\rightarrow N$ is an isomorphism from an object $M\in Fix(\mathcal{I},\eta)$, then $\eta_N=\mathcal{I}k\circ\eta_M\circ k^{-1}$ is also an isomorphism, assuring that $N\in Fix(\mathcal{I},\eta)$.

\noindent ($\Rightarrow$(only if))

If $(\mathcal{I},\eta)$ is idempotent, then any image $\mathcal{I}(C)$ belongs to $Fix(\mathcal{I},\eta)$, $C\in\mathbb{C}$ ($\eta \mathcal{I}$ is an isomorphism by hypothesis). Name $I=\mathcal{I}|_{Fix(\mathcal{I},\eta)}$ the co-restriction of the functor $\mathcal{I}$ to $Fix(\mathcal{I},\eta)$, and $H$ the full inclusion of $Fix(\mathcal{I},\eta)$ into $\mathbb{C}$. It is going to be shown that $I$ is left adjoint to $H$, with unit $\eta$.

Let $f:C\rightarrow M$ be any morphism whose codomain $M$ belongs to $Fix(\mathcal{I},\eta)$. Then, $\eta _M\circ f=\mathcal{I}f\circ \eta_C\Leftrightarrow f=(\eta^{-1}_M\circ \mathcal{I}f)\circ \eta_C$. It remains to check that $f'=\eta^{-1}_M\circ\mathcal{I}f$ is the unique morphism such that $f=f'\circ\eta_C$:
$$f''\circ\eta_C=f'\circ\eta_C\Rightarrow \mathcal{I}(f'')\circ\mathcal{I}(\eta_C)=\mathcal{I}(f')\circ\mathcal{I}(\eta_C)
$$
$$\Rightarrow \mathcal{I}(f'')\circ\eta_{\mathcal{I}(C)}=\mathcal{I}(f')\circ\eta_{\mathcal{I}(C)}\ (\mathcal{I}\eta=\eta\mathcal{I},\ by\ hypothesis)$$
$$\Rightarrow \eta_M\circ f''=\eta_M\circ f'\Rightarrow f''=f'\ (M\in Fix(\mathcal{I},\eta)).$$

\noindent ($\Leftarrow$(if))

If $\mathcal{I}(C)$ is the reflection of $C$, then $\mathcal{I}(C)\in Fix(\mathcal{I},\eta)$ and $\eta_{\mathcal{I}(C)}$ is an isomorphism. Hence, $\mathcal{I}\eta =\eta\mathcal{I}$ is an isomorphism.\end{proof}

\begin{theorem}\label{theorem:idempotent monad} Under the conditions of the previous Lemmas \ref{lemma:pointed endofunctor} and \ref{lemma:well-pointed endofunctor}, if $\mathcal{F}$ is a monad in which the unit is $\theta :1_\mathbb{C}\rightarrow \mathcal{F}$, then $(\mathcal{I},\eta)$ is idempotent, i.e., $\eta\mathcal{I}$ is an isomorphism (and $(\mathcal{I},\eta)$ is well-pointed: $\mathcal{I}\eta=\eta\mathcal{I}$).

More specifically, under the conditions of Lemma \ref{lemma:well-pointed endofunctor}, if $\mathcal{F}=GF$, with $G\vdash F:\mathbb{C}\rightarrow\mathbb{X}$ being an adjunction with unit $\theta :1_\mathbb{C}\rightarrow GF$, then there is a reflection $H\vdash I:\mathbb{C}\rightarrow \mathbb{M}$, where the right adjoint $H$ is the full inclusion of $\mathbb{M}=Fix(HI,\eta)$.

It also follows that $F\eta$ is an isomorphism.

Notice that $\mathbb{M}=Fix(\mathcal{I},\eta)$ is equal to $\mathcal{M}(F)$, the full replete subcategory of $\mathbb{C}$ determined by those objects $M$ such that $\theta_M:M\rightarrow \mathcal{F}(M)$ belongs to the right-class $\mathcal{M}$ of the prefactorization system on $\mathbb{C}$.\end{theorem}

\begin{proof}
Consider the diagram

\begin{picture}(370,60)

\put(10,40){$C$}\put(25,43){\vector(1,0){60}}
\put(52,49){$\eta_C$}
\put(95,40){$\mathcal{I}(C)$}\put(130,43){\vector(1,0){60}}
\put(152,49){$\mu_C$}\put(200,40){$GF(C)$}

\put(135,10){$f$}\put(115,33){\vector(3,-1){75}}
\put(220,20){$G(f')$}\put(215,33){\vector(0,-1){20}}

\put(200,0){$G(X)$\hspace{15pt}.}

\end{picture}\vspace{5pt}

By the universality of the unit morphism $\theta_C=\mu_C\circ\eta_C$, there exists a unique morphism $f':F(C)\rightarrow X$ such that $G(f')\circ\theta_C=f\circ\eta_C$, for every $f:\mathcal{I}(C)\rightarrow G(X)$. Then $G(f')\circ\mu_C=f$ because $\eta_C$ is an epimorphism. Hence, $\mu_C$ ($\in\mathcal{M}$) is also an universal arrow from $\mathcal{I}(C)$ to $G$, and so it must be isomorphic to $\theta_{\mathcal{I}(C)}=\mu_{\mathcal{I}(C)}\circ\eta_{\mathcal{I}(C)}$ which implies that $\eta_{\mathcal{I}(C)}$ ($\in\mathcal{E}$) is an isomorphism, by the properties of a prefactorization system $(\mathcal{E},\mathcal{M})$ (cf.\ \cite[Proposition 2.2]{CJKP:stab}).

As $F$ is a left-adjoint and every $\eta_C:C\rightarrow \mathcal{I}(C)$ is an epimorphism in $\mathbb{C}$, then $F\eta$ is an epimorphism. It is also a split monic since $F=\varepsilon F\cdot F\theta =(\varepsilon F\cdot F\mu)\cdot F\eta$, with $\varepsilon$ the counit of the adjunction $G\vdash F$. Therefore, $F\eta$ is an isomorphism.\end{proof}

\begin{remark}\label{remark:factorizarion of the Kleisli construction}

Concerning the immediately above Theorem \ref{theorem:idempotent monad}, it is easy to conclude that $\eta_{G(X)}$ is an isomorphism for all $X\in\mathbb{X}$, since $G\varepsilon_X\circ\theta_{G(X)}=(G\varepsilon_X\circ\mu_{G(X)})\circ\eta_{G(X)}=1_{G(X)}$ and $\eta_C$ is an epimorphism for all $C\in\mathbb{C}$, with $\varepsilon$ the counit of the adjunction $G\vdash F$. Hence, the restriction $G_\mathbb{M}:\mathbb{X}\rightarrow\mathbb{M}$ of the right adjoint $G$ to the codomain $\mathbb{M}$ is well defined. The left adjoint of $G_\mathbb{M}$ is obviously the restriction $F_\mathbb{M}:\mathbb{M}\rightarrow \mathbb{X}$ of $F$ to $\mathbb{M}$. As $G=H\circ G_\mathbb{M}$, then $F\cong F_\mathbb{M}\circ I$. The unit of the adjunction $F_\mathbb{M}\dashv G_\mathbb{M}$ is exactly the restriction of $\mu$ to $\mathbb{M}$, provided one chooses $\eta_M=1_M$ in every factorization $\theta_M=\mu_M\circ\eta_M$, for every $M\in\mathbb{M}$, which is equivalent to that the identity $1_\mathbb{M}$ is the counit of the induced reflection ($I(M)=M$, for every $M\in\mathbb{M}$).\end{remark}

\begin{theorem}\label{theorem:stable units idempotent monad}
Under the conditions of Theorem \ref{theorem:idempotent monad}, for a category $\mathbb{C}$ with pullbacks, the replete full reflection $I\dashv H:\mathcal{M}(F)\rightarrow\mathbb{C}$ has stable units if the following two conditions hold:\\

$(a)$ $\mathcal{F}=GF$ preserves the pullback squares of the form

\begin{picture}(210,70)
\put (60,0){$A$}\put (35,50){$A\times_{\mathcal{I}(C)}C$}
\put (145,0){$\mathcal{I}(C)$\ ;}\put(148,50){$C$}

\put(28,25){$g^*(\eta_C)$}\put(160,25){$\eta_C$}
\put(100,10){$g$}

\put(72,3){\vector(1,0){68}}\put(93,55){\vector(1,0){47}}
\put(65,45){\vector(0,-1){33}}\put(155,45){\vector(0,-1){33}}
\end{picture}\vspace{5pt}

$(b)$ $\eta_C\in\mathcal{E}'$ for every $C\in\mathbb{C}$, where $\mathcal{E}'$ stands for the largest subclass of $\mathcal{E}$ which is closed under pullbacks.\\

 Notice also that $I\dashv H:\mathcal{M}(F)\rightarrow\mathbb{C}$ does have stable units only if the condition $(a)$ above holds, for a category $\mathbb{C}$ with pullbacks and under the conditions of Theorem \ref{theorem:idempotent monad}.\end{theorem}

\begin{proof}
The final remark in the statement is going to be proved first. If the reflection $I\dashv H$ has stable units, then, by definition, the image by $\mathcal{I}$ of the pullback diagram in the statement is also a pullback diagram, meaning that both $\mathcal{I}\eta_C$ and $\mathcal{I}g^*(\eta_C)$ are isomorphisms (the former is so since $\mathcal{I}\eta =\eta\mathcal{I}$ is an isomorphism). Then, the image by $\mathcal{FI}$ of the pullback diagram in the statement is necessarily a pullback diagram. Finally, $\mathcal{F}$ preserves the pullback diagram in the statement, since $F\eta$ being an isomorphism implies $\mathcal{F}\cong\mathcal{FI}$.\\

For the if part of the theorem, consider the following commutative diagram

\begin{picture}(370,60)

\put(10,40){$P$}\put(25,43){\vector(1,0){60}}
\put(52,49){$\eta_P$}
\put(95,40){$\mathcal{I}(P)$}\put(130,43){\vector(1,0){60}}
\put(152,49){$\mu_P$}\put(200,40){$\mathcal{F}(P)$}

\put(15,33){\vector(0,-1){20}}\put(20,20){$g^*(\eta_C)$}
\put(110,20){$\mathcal{I}g^*(\eta_C)$}\put(105,33){\vector(0,-1){20}}
\put(220,20){$\mathcal{F}g^*(\eta_C)$}\put(215,33){\vector(0,-1){20}}

\put(10,0){$A$}\put(25,3){\vector(1,0){60}}
\put(52,9){$\eta_{A}$}
\put(95,0){$\mathcal{I}(A)$}\put(130,3){\vector(1,0){60}}
\put(152,9){$\mu_A$}\put(200,0){$\mathcal{F}(A)$,}

\end{picture}\vspace{5pt}
\noindent where $P=A\times_{\mathcal{I}(C)}C$. We need to show that $\mathcal{I}g^*(\eta_C)$ is an isomorphism, or equivalently that $\mathcal{I}g^*(\eta_C)\in \mathcal{E}\cap\mathcal{M}$, using the properties of prefactorization systems (cf.\ \cite[Proposition 2.2]{CJKP:stab}): as $\eta_C\in \mathcal{E}'$, then $\eta_A\circ g^*(\eta_C)=\mathcal{I}g^*(\eta_C)\circ\eta_P\in \mathcal{E}$, implying that $\mathcal{I}g^*(\eta_C)\in\mathcal{E}$; as $\mathcal{F}g^*(\eta_C)$ is an isomorphism (by hypothesis $(a)$, since $F\eta$ is an isomorphism) then $\mu_A\circ\mathcal{I}g^*(\eta_C)\cong\mu_P\in\mathcal{M}$, implying that $\mathcal{I}g^*(\eta_C)\in\mathcal{M}$ ($\mu_A\in\mathcal{M}$).\end{proof}

\begin{remark}\label{remark:mistake}
We take the opportunity here to report a mistake in our paper \cite{X:well behaved}. Theorem 2.4 in \cite{X:well behaved} corresponds to Theorem \ref{theorem:stable units idempotent monad} just above, but in the latter the two conditions $(a)$ and $(b)$ are now only sufficient for the derived reflection to have stable units. Fortunately, this mistake had no impact on the other results of \cite{X:well behaved}, since only the sufficiency was used.
\end{remark}

\begin{examples}\label{example:regular cats} Consider any adjunction $G\vdash F:\mathbb{C}\rightarrow \mathbb{X}$ with unit $\theta :1_\mathbb{C}\rightarrow GF$, in which $\mathbb{C}$ is a regular category (i.e., it is finitely complete, it has coequalizers, and the class of its regular epimorphisms is closed under pullbacks) and the left adjoint $F$ preserves pullbacks. Then, there is a reflection $H\vdash I:\mathbb{C}\rightarrow Mono(F)$ with stable units, where $Mono(F)$ is a full (replete) subcategory of $\mathbb{C}$ determined by the objects $M\in\mathbb{C}$ for which the unit morphism $\theta_M:M\rightarrow GF(M)$ is a monomorphism. The unit $\eta :1_\mathbb{C}\rightarrow HI$ consists componentwise of regular epimorphisms of $\mathbb{C}$; in fact, $(\mathcal{E},\mathcal{M})=(Regular Epimorphisms,\ Monomorphisms)$ is an (orthogonal) factorization system on $\mathbb{C}$.
\end{examples}

\subsection{Derivable very-well-behaved epireflections}\label{subsec:derivable very-well-behaved epireflections}

\begin{theorem}\label{theorem:stable units+enough edm}
Under the conditions of Theorem \ref{theorem:idempotent monad}, if the derived reflection $H\vdash I:\mathbb{C}\rightarrow \mathcal{M}(F)$ has stable units, and for every object $C(\in\mathbb{C})$ there is an effective descent morphism $p:M\rightarrow C$ (in $\mathbb{C}$) with domain in $\mathcal{M}(F)$, then there is a monotone-light factorization system $(\mathcal{E}'_I,\mathcal{M}^*_I)$ on $\mathbb{C}$, obtained by simultaneously stabilizing and localizing the reflective factorization system $(\mathcal{E}_I,\mathcal{M}_I)$.
\end{theorem}
 \begin{proof}
 This small result (but crucial in this work) is included in Corollary 6.2 in \cite{X:iml}, which follows from the main result of \cite{CJKP:stab}.
 \end{proof}

 \begin{examples}\label{example:Barr exact cats}
 Under the conditions of Examples \ref{example:regular cats}, let $\mathbb{C}$ be not only regular, but furthermore (Barr) exact, which is known to imply that in $\mathbb{C}$ the effective descent morphisms are the regular epimorphisms (cf.\ \cite{JST:edm}). Suppose also that $\mathbb{C}$ is cocomplete, and that there is a full subcategory $\mathbb{N}$ of $\mathbb{C}$ which is dense in $\mathbb{C}$ (i.e., every $C\in\mathbb{C}$ is a colimit of objects of $\mathbb{N}$) and closed in $Mono(F)(=\mathcal{M}(F))$ for coproducts in $\mathbb{C}$ (hence, the canonical presentation of every $C\in\mathbb{C}$ as a colimit gives $p:M\rightarrow C$, with $M=\coprod_{i\in I}N_i\in Mono(F)$ and $p$ a regular epimorphism, that is, an effective descent morphism).

 Then, by Theorems \ref{theorem:stable units idempotent monad} and \ref{theorem:stable units+enough edm}, it follows that the induced reflection, from the left adjoint functor $F$ preserving pullbacks, produces a monotone-light factorization system on $\mathbb{C}$.

 Remark that, by next Lemma \ref{lemma:ob(N) generating set for C}$(b)$, instead of demanding that the dense subcategory $\mathbb{N}$ is closed in $Mono(F)$ for coproducts, one could alternatively ask $\mathbb{N}$ to be a dense subcategory of $\mathbb{C}$ closed under coproducts and such that the restriction of the left-adjoint $F$ to $\mathbb{N}$ is faithful.
 \end{examples}

  A set $S$ of objects of a category $\mathbb{C}$ is a generating set if, for every two distinct parallel morphisms $f,f'\in\mathbb{C}$, there exists an object $N\in S$ and a morphism $g$ with domain $N$ such that $f\circ g\neq f'\circ g$.

  Notice that the objects of a full subcategory $\mathbb{N}$ which is dense in $\mathbb{C}$ are obviously a generating set for $\mathbb{C}$.\\

The following Lemma \ref{lemma:ob(N) generating set for C} is needed in the proof of next Theorems \ref{theorem:coproduct of unit morphisms} and \ref{theorem:ob(N) generating set for C}, which give conditions under which an adjunction from a category of presheaves has \emph{enough} effective descent morphisms, in the sense of Theorem \ref{theorem:stable units+enough edm}.

\begin{lemma}\label{lemma:ob(N) generating set for C}
Consider any adjunction $G\vdash F:\mathbb{C}\rightarrow \mathbb{X}$ with unit $\theta:1_\mathbb{C}\rightarrow GF$.

Suppose there is a full subcategory $\mathbb{N}$ of $\mathbb{C}$ whose set of objects $ob(\mathbb{N})$ is a generating set for $\mathbb{C}$.

\noindent Then:\vspace{3pt}

\noindent $(a)$ for any object $C\in\mathbb{C}$, the unit morphism $\theta_C:C\rightarrow GF(C)$ is a monomorphism if and only if, for every object $N\in\mathbb{N}$, the restrictions $F_{N,C}: \mathbb{C}(N,C)\rightarrow \mathbb{X}(F(N),F(C))$ of the functor $F$ are all injective;\vspace{3pt}

\noindent $(b)$\cite[Lemma 4.2]{X:concordant&monotone} each unit morphism $\theta_N:N\rightarrow GF(N)$ is a monomorphism, $N\in\mathbb{N}$, if and only if the restriction of the left -adjoint $F$ to $\mathbb{N}$ is faithful.
\end{lemma}
\begin{proof}
\noindent $(a)$\vspace{3pt}

\noindent $(\Rightarrow)$ $Ff=Ff':F(N)\rightarrow F(C) \Rightarrow GFf\circ\theta_N=GFf'\circ\theta_N$

$\Rightarrow\theta_C\circ f=\theta_C\circ f'$ ($\theta$ is a natural transformation)

$\Rightarrow f=f'$ ($\theta_C$ is a monomorphism by hypothesis)

\noindent (in fact, it was proved that $F_{N,C}$ is injective for any $N\in\mathbb{C}$, provided $\theta_C$ is a monomorphism, and not only for $N\in\mathbb{N}$).\vspace{3pt}

\noindent $(\Leftarrow)$ Suppose, by \emph{reductio ad absurdum}, that there are $g\neq g':C'\rightarrow C$ such that $\theta_C\circ g=\theta_C\circ g'$. Being $ob(\mathbb{N})$ a generating set for $\mathbb{C}$, there exists $\tau :N\rightarrow C'$ with $N\in\mathbb{N}$, such that $g\circ\tau\neq g'\circ\tau$. Then, $\theta_C\circ g\circ\tau=\theta_C\circ g'\circ\tau$ implies that $F(g\circ\tau)=F(g'\circ\tau)$, because $\theta_N$ is a unit morphism, and so $g\circ\tau =g'\circ\tau$ by hypothesis, which is a contradiction.\vspace{3pt}

\noindent $(b)$ The proof of this item follows immediately from item $(a)$ by just considering all $C\in \mathbb{N}$.\end{proof}

The following general result of Lemma \ref{lemma:unit of coproduct = coproduct of units} is needed in the proof of the consecutive Theorem \ref{theorem:coproduct of unit morphisms}.

\begin{lemma}\label{lemma:unit of coproduct = coproduct of units}
Consider an adjunction $G\vdash F:\mathbb{C}\rightarrow \mathbb{X}$ with unit $\theta:1_\mathbb{C}\rightarrow GF$. Suppose that, for the family $(T_i)_{i\in I}$ of objects of $\mathbb{C}$, there exists its coproduct $\coprod_{i\in I}T_i$ in $\mathbb{C}$. Then, $$\theta_{\coprod_{i\in I}T_i}=\coprod_{i\in I}\theta_{T_i},$$ provided the right-adjoint $G$ preserves the coproduct $\coprod_{i\in I}F(T_i)$ in $\mathbb{X}$.
\end{lemma}
\begin{proof}
Let $\iota_{T_i}:T_i\rightarrow \coprod_{i\in I}T_i$, $\iota_{F(T_i)}=F\iota_{T_i}:F(T_i)\rightarrow \coprod_{i\in I}F(T_i)=F(\coprod_{i\in I}T_i)$ and $\iota_{GF(T_i)}=G\iota_{F(T_i)}=GF\iota_{T_i}:GF(T_i)\rightarrow \coprod_{i\in I}GF(T_i)=G\coprod_{i\in I}F(T_i)=GF(\coprod_{i\in I}T_i)$ denote the three coproduct injections in question (recall that $F$ preserves colimits, because it is a left adjoint).

Since $\theta$ is a natural transformation, $$\theta_{\coprod_{i\in I}T_i}\circ\iota_{T_i}=GF(\iota_{T_i})\circ\theta_{T_i}.$$

By definition of $\coprod_{i\in I}\theta_{T_i}$, $$\coprod_{i\in I}\theta_{T_i}\circ\iota_{T_i}=\iota_{GF( T_i)}\circ\theta_{T_i}.$$

Hence, as $GF(\iota_{T_i})=\iota_{GF( T_i)}$, $\theta_{\coprod_{i\in I}T_i}\circ\iota_{T_i}=\coprod_{i\in I}\theta_{T_i}\circ\iota_{T_i}$ for all $i\in I$, implying that $\theta_{\coprod_{i\in I}T_i}=\coprod_{i\in I}\theta_{T_i}$ by the definition of a coproduct $\coprod_{i\in I}T_i$.\end{proof}

It is convenient now to recall that, in a category of presheaves $Set^{\mathbb{D}^{op}}$, the representable functors $\mathbb{D}(-,D)$ constitute a generating set, $D\in\mathbb{D}$ (cf.\ \cite[\S III.7]{SM:cat}.

\begin{theorem}\label{theorem:coproduct of unit morphisms}\cite[See Remark 6.1]{X:concordant&monotone} Consider any adjunction $G\vdash F:\mathbb{C}\rightarrow \mathbb{X}$ with unit $\theta:1_\mathbb{C}\rightarrow GF$, where $\mathbb{C}=Set^{\mathbb{D}^{op}}$ is a category of presheaves:\vspace{3pt}

\noindent if\vspace{3pt}

$\bullet$ $F\circ y$ is faithful

(where $y:\mathbb{D}\rightarrow Set^{\mathbb{D}^{op}}$ is the Yoneda embedding) and

$\bullet$ $G$ preserves coproducts,\vspace{3pt}

\noindent then\vspace{3pt}

the unit morphism $\theta_{\coprod_{i\in I}\mathbb{D}(-,D_i)}$ is a monomorphism in $\mathbb{C}=Set^{\mathbb{D}^{op}}$,

for each family $(D_i)_{i\in I}$ of objects of $\mathbb{D}$.
\end{theorem}

\begin{proof}
Since $G$ preserves coproducts, $\theta_{\coprod_{i\in I}T_i}=\coprod_{i\in I}\theta_{T_i}$ , for each family $(T_i)_{i\in I}$ of objects of $\mathbb{C}$ (cf.\ Lemma \ref{lemma:unit of coproduct = coproduct of units}). Hence, $\theta_{\coprod_{i\in I}\mathbb{D}(-,D_i)}$ is a monomorphism in $\mathbb{C}=Set^{\mathbb{D}^{op}}$, for each family $(D_i)_{i\in I}$ of objects of $\mathbb{D}$, by Lemma \ref{lemma:ob(N) generating set for C}$(b)$ and the fact that a monomorphism in $Set^{\mathbb{D}^{op}}$ is a componentwise injection.
\end{proof}

The previous Theorem \ref{theorem:coproduct of unit morphisms} will be helpful in section \ref{sec:left Kan extensions}. The following Corollary \ref{corollary:ob(N) generating set for C and edm} of Theorem \ref{theorem:ob(N) generating set for C} is needed in section \ref{sec:right Kan extensions}.

\begin{theorem}\label{theorem:ob(N) generating set for C}
Consider any adjunction $G\vdash F:\mathbb{C}\rightarrow \mathbb{X}$ with unit $\theta:1_\mathbb{C}\rightarrow GF$, where $\mathbb{C}=Set^{\mathbb{D}^{op}}$ is a category of presheaves:\vspace{3pt}

\noindent if\vspace{3pt}

$\bullet$ $F\circ y$ is faithful and injective on objects

(where $y:\mathbb{D}\rightarrow Set^{\mathbb{D}^{op}}$ is the Yoneda embedding),

$\bullet$ the coproduct injections in $\mathbb{X}$ are monomorphisms,

$\bullet$ the equality $\iota_j\circ f=\iota_{j'}\circ f'$ implies $\iota_j=\iota_{j'}$,

 for every two coproduct injections

 $\iota_j:X_j\rightarrow\coprod_{i\in I}X_i$ and $\iota_{j'}:X_{j'}\rightarrow\coprod_{i\in I}X_i$ in $\mathbb{X}$,

 and any pair of morphisms $f,f'$ with common domain,\vspace{3pt}

\noindent then\vspace{3pt}

the unit morphism $\theta_{\coprod_{i\in I}\mathbb{D}(-,D_i)}$ is a monomorphism in $\mathbb{C}=Set^{\mathbb{D}^{op}}$,

for each family $(D_i)_{i\in I}$ of objects of $\mathbb{D}$.\end{theorem}

\begin{proof}
  The representable functors $\mathbb{D}(-,D)$ in $Set^{\mathbb{D}^{op}}$ are known to constitute a dense full subcategory of $Set^{\mathbb{D}^{op}}$ (cf.\ \cite[Theorem 1 in \S III.7]{SM:cat}), therefore a generating set for $Set^{\mathbb{D}^{op}}$. By Lemma \ref{lemma:ob(N) generating set for C}$(a)$, we have just to show that, for every family $(D_i)_{i\in I}$ of objects in $\mathbb{D}$, $F_{\mathbb{D}(-,D),\coprod_{i\in I}\mathbb{D}(-,D_i)}:$

 $Set^{\mathbb{D}^{op}}(\mathbb{D}(-,D),\coprod_{i\in I}\mathbb{D}(-,D_i))\rightarrow \mathbb{X}(F(\mathbb{D}(-,D)),F(\coprod_{i\in I}\mathbb{D}(-,D_i)))$ is an injection, for every $D\in\mathbb{D}$.\vspace{3pt}

It is easy to show, using the Yoneda Lemma, that for any $\alpha :\mathbb{D}(-,D)\rightarrow\coprod_{i\in I}\mathbb{D}(-,D_i)$ there exists $\alpha_j:\mathbb{D}(-,D)\rightarrow\mathbb{D}(-,D_j)$, with $j\in I$, such that $\delta_j\circ\alpha_j=\alpha$, where $\delta_j:\mathbb{D}(-,D_j)\rightarrow\coprod_{i\in I}\mathbb{D}(-,D_i)$ is a coproduct injection, $j\in I$ ($\alpha_j$ is determined by $\alpha_D(1_D):D\rightarrow D_j$, i.e., $\alpha_{j,D}(1_D)=\alpha_D(1_D)$).\vspace{3pt}

Let $\alpha' :\mathbb{D}(-,D)\rightarrow\coprod_{i\in I}\mathbb{D}(-,D_i)$ be another morphism in $Set^{\mathbb{D}^{op}}$ such that $F\alpha'=F\alpha$; then, there exists $\alpha'_{j'}:\mathbb{D}(-,D)\rightarrow\mathbb{D}(-,D_{j'})$ such that $\delta_{j'}\circ\alpha'_{j'}=\alpha'$, where $\delta_{j'}:\mathbb{D}(-,D_{j'})\rightarrow\coprod_{i\in I}\mathbb{D}(-,D_i)$ is a coproduct injection ($j'\in I$); hence, $F\delta_j\circ F\alpha_j=F\delta_{j'}\circ F\alpha'_{j'}$. Notice that $\iota_j=F\delta_j$ and $\iota_{j'}=F\delta_{j'}$ are coproduct injections in $\mathbb{X}$ (since $F$ is a left-adjoint). By the hypotheses in the statement, $\iota_j=\iota_{j'}$, which implies $\delta_j=\delta_{j'}$ ($F\delta_j=F\delta_{j'}\Rightarrow F(\mathbb{D}(-,D_j))=F(\mathbb{D}(-,D_{j'}))\Rightarrow\mathbb{D}(-,D_j)=\mathbb{D}(-,D_{j'})$, because $Fy$ is injective on objects) and $F\alpha_j=F\alpha'_{j'}$ ($\iota_j$ is a monomorphism by hypothesis). Then, $\alpha_j=\alpha_{j'}$ because $Fy$ is faithful (it was concluded before that $\alpha_j$ and $\alpha'_{j'}$ are parallel morphisms in the dense subcategory of representable functors). Finally, $\alpha=\delta_j\circ\alpha_j=\delta_{j'}\circ\alpha'_{j'}=\alpha'$.\end{proof}

\begin{corollary}\label{corollary:ob(N) generating set for C and edm}\cite[See footnote in Remark 6.1]{X:concordant&monotone} If both $\mathbb{C}=Set^{\mathbb{D}^{op}}$ and $\mathbb{X}=Set^\mathbb{B}$ are categories of presheaves, and $F\circ y$ is faithful and injective on objects (where $y:\mathbb{D}\rightarrow Set^{\mathbb{D}^{op}}$ is the Yoneda embedding), then the unit morphism $\theta_{\coprod_{i\in I}\mathbb{D}(-,D_i)}$ is a monomorphism in $\mathbb{C}=Set^{\mathbb{D}^{op}}$, for each family $(D_i)_{i\in I}$ of objects of $\mathbb{D}$.
\end{corollary}
\begin{proof}
First, it is trivial that the counit injections are monomorphisms in $Set^\mathbb{B}$.\vspace{3pt}

According to Theorem \ref{theorem:ob(N) generating set for C}, if $\iota_j\circ f=\iota_{j'}\circ f'$ implies that $\iota_j=\iota_{j'}$ in $Set^\mathbb{B}$ (where $\iota_j:X_j\rightarrow\coprod_{i\in I}X_i$ and $\iota_{j'}:X_{j'}\rightarrow\coprod_{i\in I}X_i$ are coproduct injections in $Set^\mathbb{B}$, for any family $(X_i)_{i\in I}$ of objects in $Set^\mathbb{B}$) then the statement would have been proved:

 in fact, for every $B\in \mathbb{B}$, $\iota_{j,B}\circ f_B=\iota_{j',B}\circ f'_B$ implies that $X_j(B)=X_{j'}(B)$ since $(\coprod_{i\in I}X_i)(B)=\coprod_{i\in I}X_i(B)$ and $\iota_{j,B}$, $\iota_{j',B}$ are inclusions in $Set$; therefore, $\iota_{j,B}=\iota_{j',B}$ for every $B\in\mathbb{B}$, that is, $\iota_j=\iota_{j'}$.
\end{proof}

\subsection{Uplifting factorization systems to functor categories}\label{subsec:Uplifting factorization systems to functor categories}

The last result in this section \ref{sec-General Results} allows us to start with a factorization of the morphisms in a category $\mathbb{S}$, and then raise it to the factorization of the natural transformations in the functor category $\mathbb{S}^\mathbb{I}$.

\begin{lemma}\label{lemma:uplifting (pre)factsys}
Let $\mathcal{S}$ be a class of morphisms in the category $\mathbb{S}$. Define $\mathcal{S}^\mathbb{I}$ as the class of all morphisms in the functor category $\mathbb{S}^\mathbb{I}$ whose components are in $\mathcal{S}$ ($\alpha=(\alpha_i)_{i\in \mathbb{I}}\in\mathcal{S}^\mathbb{I}$ if and only if $\alpha_i$ is in $\mathcal{S}$ for all $i\in \mathbb{I}$).

$(a)$ If $(\mathcal{E},\mathcal{M})$ is a prefactorization system on $\mathbb{S}$, then there is a prefactorization system $(\mathcal{F},\mathcal{N})$ on $\mathbb{S}^\mathbb{I}$ such that $\mathcal{E}^\mathbb{I}\subseteq\mathcal{F}$ and $\mathcal{M}^\mathbb{I}\subseteq\mathcal{N}$.

$(b)$ If $(\mathcal{E},\mathcal{M})$ is a factorization system on $\mathbb{S}$, then $(\mathcal{E}^\mathbb{I},\mathcal{M}^\mathbb{I})$ is a factorization system on $\mathbb{S}^\mathbb{I}$.
\end{lemma}

\begin{proof} $(a)$ It has to be proved that, for any commutative square $v\circ e=m\circ u$ in $\mathbb{S}^\mathbb{I}$, with $e:A\rightarrow B\in\mathcal{E}^\mathbb{I}$ and $m:C\rightarrow D\in\mathcal{M}^\mathbb{I}$, there is one and only one morphism $w:B\rightarrow C$ such that $w\circ e=u$ and $m\circ w=v$.

If such a morphism $w$ exists then it must be the unique $w=(w_i)_{i\in\mathbb{I}}$ such that $w_i\circ e_i=u_i$ and $m_i\circ w_i=v_i$, for every $i\in \mathbb{I}$. The question is if $(w_i)_{i\in \mathbb{I}}$ is a natural transformation.

Let $f:i\rightarrow j$ be a morphism in $\mathbb{I}$. It has to be shown that $Cf\circ w_i=w_j\circ Bf$, or equivalently (since $e_i\downarrow m_j$) $Cf\circ w_i\circ e_i=Cf\circ u_i=w_j\circ Bf\circ e_i$ and $m_j\circ Cf\circ w_i=v_j\circ Bf=m_j\circ w_j\circ Bf$:

\noindent $Cf\circ (w_i\circ e_i)=Cf\circ u_i$, $w_j\circ (Bf\circ e_i)=(w_j\circ e_j)\circ Af=u_j\circ Af=Cf\circ u_i$;

\noindent $(m_j\circ Cf)\circ w_i=Df\circ (m_i\circ w_i)=Df\circ v_i=v_j\circ Bf$, $(m_j\circ w_j)\circ Bf=v_j\circ Bf$.

$(b)$ Being $(\mathcal{E},\mathcal{M})$ a factorization system on $\mathbb{S}$, it is obvious that both $\mathcal{E}^\mathbb{I}$ and $\mathcal{M}^\mathbb{I}$ contain the identities and are closed under composition with isomorphisms. By $(a)$, in order that $(\mathcal{E}^\mathbb{I},\mathcal{M}^\mathbb{I})$ is a factorization system on $\mathbb{S}^\mathbb{I}$, it remains to show that each morphism $\alpha :A\rightarrow B$ in $\mathbb{S}^\mathbb{I}$ factorises as $\alpha =m\circ e$, with $m\in\mathcal{M}^\mathbb{I}$ and $e\in\mathcal{E}^\mathbb{I}$ (cf.\ \cite[\S 2.8]{CJKP:stab}):

\begin{picture}(370,60)

\put(-3,40){$A(i)$}\put(25,43){\vector(1,0){60}}
\put(52,49){$e_i$}
\put(95,40){$C(i)$}\put(130,43){\vector(1,0){60}}
\put(152,49){$m_i$}\put(200,40){$B(i)$}\put(250,40){$m_i\circ e_i=\alpha_i$}

\put(12,33){\vector(0,-1){20}}\put(-8,20){$Af$}
\put(85,20){$Cf$}\put(105,33){\vector(0,-1){20}}
\put(195,20){$Bf$}\put(215,33){\vector(0,-1){20}}

\put(-3,0){$A(j)$}\put(25,3){\vector(1,0){60}}
\put(52,9){$e_j$}
\put(95,0){$C(j)$}\put(130,3){\vector(1,0){60}}
\put(152,9){$m_j$}\put(200,0){$B(j)$}\put(250,0){$m_j\circ e_j=\alpha_j$,}

\end{picture}\vspace{5pt}

\noindent where $C:\mathbb{I}\rightarrow\mathbb{S}$ being a functor follows from $(\mathcal{E},\mathcal{M})$ being a (pre)factorization system.\end{proof}

\section{The case of left Kan extensions for presheaves}\label{sec:left Kan extensions}

\begin{theorem}\label{theorem:leftKanpresheavesverywellbehaved} Consider a functor $K:\mathbb{B}\rightarrow\mathbb{A}$, such that $\mathbb{B}$ is a small category and $\mathbb{A}$ is a category with small hom-sets. Then, there is an adjunction $Set^K\vdash Lan_K:Set^\mathbb{B}\rightarrow Set^\mathbb{A}$ given by left Kan extensions.

If the left-adjoint $Lan_K$ preserves pullbacks and $Lan_K\circ y$ is faithful (where $y:\mathbb{B}^{op}\rightarrow Set^\mathbb{B}$ is the Yoneda embedding), then there is an induced reflection with stable units $H\vdash I:Set^\mathbb{B}\rightarrow Mono(Lan_K)$ and a monotone-light factorization $(\mathcal{E}'_I,\mathcal{M}^*_I)$ on $Set^\mathbb{B}$.
\end{theorem}

\begin{proof}
The statement is a special case of Examples \ref{example:Barr exact cats}, with $\mathbb{C}=Set^\mathbb{B}$, $\mathbb{X}=Set^\mathbb{A}$, $(\mathcal{E},\mathcal{M})=(Epis, Monos)$ is the lifting given by Lemma \ref{lemma:uplifting (pre)factsys}$(b)$ of the factorization system $(Surjections,Injections)$ on $Set$.

In fact, the representable functors $\mathbb{B}(B,-)$ in $Set^\mathbb{B}$ constitute a dense subcategory of $Set^\mathbb{B}$, whose unit morphisms $\theta_{\mathbb{B}(B,-)}:\mathbb{B}(B,-)\rightarrow Lan_K(\mathbb{B}(B,-))\circ K$ are by assumption monomorphisms ($\Leftrightarrow Lan_K\circ y$ faithful) and their coproducts are also in $Mono(Lan_K)$ by Theorem \ref{theorem:coproduct of unit morphisms} (the right adjoint $Set^K$ preserves colimits).\end{proof}

\begin{remark}\label{remark:LanK preserves pullbacks too strong}
In Theorem \ref{theorem:leftKanpresheavesverywellbehaved} just above, instead of asking that $Lan_K$ preserves all pullbacks, it could be only demanded the weaker condition that $Set^K\circ Lan_K$ preserves pullbacks (in fact, only the ones given in Theorem \ref{theorem:stable units idempotent monad}$(a)$).
\end{remark}

Notice that, if $K=\ !:\mathbb{J}\rightarrow\mathbf{1}$ is the functor whose codomain is the category $\mathbf{1}$ (one object and no non-identity morphisms) and the domain $\mathbb{J}$ is small, then the adjunction $Set^K\vdash Lan_K:Set^\mathbb{J}\rightarrow Set$ is $\Delta\vdash Colim$, that is, the right adjoint is the diagonal functor and the left adjoint takes colimits.

Recall, as well, that a category $\mathbb{J}=\coprod_{i\in I}\mathbb{J}_i$ is pseudo-filtered if its connected components $\mathbb{J}_i$ are filtered (cf.\ Exercise 2 in \cite[\S IX.2]{SM:cat}).

\begin{lemma}\label{lemma:colim functor}
Consider the adjunction $\Delta\vdash Colim: Set^\mathbb{J}\rightarrow Set$, with unit $\theta$, from which derives the full reflective subcategory $Mono(Colim)$ as in Theorem \ref{theorem:idempotent monad}, using the factorization system $(Epis,Monos)=(Surjections^\mathbb{J},Injections^\mathbb{J})$ (cf.\ Lemma \ref{lemma:uplifting (pre)factsys}$(b)$). Let the category $\mathbb{J}$ be pseudo-filtered (cf.\ last paragraph just before this Lemma \ref{lemma:colim functor}). Then, the two following statements are valid.

\noindent $(a)$ $M\in Mono(Colim)$ if and only if $Mu$ is injective, for every $u$ in $\mathbb{J}$.

\noindent $(b)$ The functor $Colim: Set^\mathbb{J}\rightarrow Set$ preserves pullbacks.
\end{lemma}

\begin{proof}
$(a)$

\noindent $(\Rightarrow)$ $M\in Mono(Colim)\Leftrightarrow \forall_{j\in\mathbb{J}}\ \theta_{M,j}:M(j)\rightarrow Colim(M)$ is an injection, where $\theta_{M,j}$ is the $j$-component of the unit morphism of $M\in Set^\mathbb{J}$ (the colimit cocone of $M:\mathbb{J}\rightarrow Set$), then

$\forall_{u:j\rightarrow j'} \forall_{x,x'\in M(j)} Mu(x)=Mu(x')\Rightarrow x=x'$.\vspace{3pt}

\noindent In fact, $\theta_{M,j}(x)=\theta_{M,j'}\circ Mu(x)=\theta_{M,j'}\circ Mu(x'
)=\theta_{M,j}(x')$, and so $x=x'$ because $\theta_{M,j}$ is an injection:

$$\begin{picture}(55,55)(0,0)

\put(35,53){$Colim(M)$}
\put(8,35){$\theta_{M,j}$}\put(80,35){$\theta_{M,j'}$}

\put(10,12){\vector(1,1){35}}\put(100,12){\vector(-1,1){35}}

\put(0,0){$M(j)$}\put (100,0){$M(j')$\hspace{10pt}.}

\put(30,2){\vector(1,0){65}}\put(45,6){$Mu$}

\end{picture}$$


\noindent $(\Leftarrow )$ $Colim(M)=\coprod_{j\in \mathbb{J}}M(j)/E$, where $E$ is the equivalence relation which is about to be described. Let $x\in M(j)$ and $x'\in M(j')$, then $xEx'$ if and only if there exists $u:j\rightarrow k$ and $u':j'\rightarrow k$ such that $Mu(x)=Mu'(x')$ (cf.\ the proof of Theorem 1 in \cite[\S IX.2]{SM:cat}, where this result is stated; the generalization to $\mathbb{J}=\coprod_{i\in I}\mathbb{J}_i$ pseudo-filtered is obvious, since the colimits are obtained as the disjoint union of the colimits of the filtered connected components $\mathbb{J}_i$).

Let $x,x'\in M(j)$ such that $\theta_{M,j}(x)=\theta_{M,j}(x')$, and therefore $xEx'$. Then, as $j=j'$, it follows that
\begin{picture}(105,16)(0,5)

\put (0,4){$\exists j$}\put (50,4){$k$}\put (95,4){$k'$}

\put(15,13){\vector(1,0){30}}\put(25,15){$u$}
\put(15,3){\vector(1,0){30}}\put(25,4){$u'$}
\put(60,8){\vector(1,0){30}}\put(70,9){$w$}

\end{picture}\ \ such that $M(w\circ u)(x)=M(w\circ u')(x')$, which implies by the hypothesis that $x=x'$, since $w\circ u=w\circ u'$ (the existence of such $w:k\rightarrow k'$ is assured by $\mathbb{J}$ being pseudo-filtered: cf.\ Exercise 2 in \cite[\S IX.2]{SM:cat}, where the definition of pseudo-filtered category is displayed).

$(b)$ It is well known that finite limits commute with filtered colimits in $Set$. The statement is just a refinement of this fact (cf.\ Exercise 4 in \cite[\S IX.2]{SM:cat}): pseudo-filtered colimits commute with pullbacks in $Set$; it follows that if $\mathbb{J}$ is pseudo-filtered then $Colim$ preserves pullbacks.\end{proof}

\begin{theorem}\label{theorem:Colimverywellbehaved}

Under the conditions of Lemma \ref{lemma:colim functor}, if every morphism in $\mathbb{J}$ is a monomorphism, then there is a monotone-light factorization system $(\mathcal{E}'_I,\mathcal{M}^*_I)$ on $Set^\mathbb{J}$.

\end{theorem}
\begin{proof}

By Lemma \ref{lemma:colim functor}$(b)$, the left-adjoint $Lan_!=Colim$ preserves pullbacks. According to Theorem \ref{theorem:leftKanpresheavesverywellbehaved}, there exists  the required monotone-light factorization, provided any representable functor $\mathbb{J}(i,-)$ belongs to $Mono(Colim)$. This is the case if, for every $u:j\rightarrow j'$ in $\mathbb{J}$, the image $\mathbb{J}(i,-)u:\mathbb{J}(i,-)(j)\rightarrow \mathbb{J}(i,-)(j')$ is always an injection, by Lemma \ref{lemma:colim functor}$(a)$. It is obviously so, since the functions (compose with $u$) $u_*:\mathbb{J}(i_k,j)\rightarrow\mathbb{J}(i_k,j')$ are injections if $u$ is a monomorphism.\end{proof}

\begin{examples}\label{example:colimit}Any preorder $\mathbb{J}$, such that each of its connected components is a filtered category (also called filtered set or directed set), gives rise to a very-well-behaved induced reflection $Set^\mathbb{J}\rightarrow Mono(Colim)$.

 For instance, consider the very simple special case of Theorem \ref{theorem:Colimverywellbehaved} with $\mathbb{J}=\mathbf{2}$, the ordinal number with two objects $0$, $1$ and only one non-identity morphism $u:0\rightarrow 1$. As $\mathbf{2}$ is a filtered preorder, the reflection $H\vdash I:Set^\mathbf{2}\rightarrow Mono(Colim)$ is \emph{very-well-behaved}, that is, has stable units and a monotone-light factorization $(\mathcal{E}'_I,\mathcal{M}^*_I)$, obtained by simultaneous stabilization and localization of the classes in the reflective factorization system $(\mathcal{E}_I,\mathcal{M}_I)$. It is going to be shown that this particular monotone-light factorization $(\mathcal{E}'_I,\mathcal{M}^*_I)$ is non-trivial, that is, $\mathcal{E}'_I$ is strictly contained in $\mathcal{E}_I$ ($\Leftrightarrow\mathcal{M}^*_I$ strictly contains $\mathcal{M}_I$).

The unit morphism $\theta_M:M\rightarrow\Delta Colim(M)=\Delta M(1)$ consists in two functions $\theta_{M,0}=Mu$ and $\theta_{M,1}=id_{M(1)}$, for all $Mu=M(0)\rightarrow M(1)$. Hence, from the $(Surjections, Injections)$-factorizations $$\theta_{M,0}=\mu_{M,0}\circ\eta_{M,0},\ \theta_{M,1}=1_{M(1)}=\mu_{M,1}\circ\eta_{M,1}=1_{M(1)}\circ 1_{M(1)},$$ \noindent it is easy to conclude that the replete full subcategory $Mono(Colim)$ is determined by all injections (cf.\ Lemma \ref{lemma:colim functor}).
Concluding then that a morphism $\varphi =(\varphi_0,\varphi_1):M\rightarrow N$ belongs to $\mathcal{E}_I$ if and only if $\varphi_1:M(1)\rightarrow N(1)$ is a bijection and $(Nu\circ\varphi_0)(M(0))=Nu(N(0))$ (where $Nu:N(0)\rightarrow N(1)$ and $\varphi_0:M(0)\rightarrow N(0)$):

\noindent indeed, $\varphi =(\varphi_0,\varphi_1)\in\mathcal{E}_I$ if and only if $I\varphi =I(\varphi_0,\varphi_1)=(I(\varphi)_0,I(\varphi)_1)$ is an isomorphism (cf.\ \cite[(3.2) in 3.1]{CJKP:stab}), that is both components are bijections. Noticing that $I(\varphi)_1=\varphi_1$, and that $\varphi_1$ being a bijection implies that $I(\varphi)_0$ is an injection, it remains to characterize the surjection of $I(\varphi)_0$ in terms of the initial data: $I(\varphi)_0$ is a surjection iff $I(\varphi)_0\circ\eta_{M,0}$ is a surjection iff $\eta_{N,0}\circ\varphi_0$ is a surjection iff $\eta_{N,0}\circ\varphi_0(M(0))=I(N)(0)$ iff $Nu\circ\varphi_0(M(0))=Nu(N(0))(\cong I(N)(0))$.

Let $M=N$, $M(0)=\{x,y\}$, $M(1)=\{x\}$ and $\varphi_0(x)=x=\varphi_0(y)$. This particular morphism $\varphi =(\varphi_0,\varphi_1)$ belongs to $\mathcal{E}_I$ but does not belong to its largest subclass closed under pullbacks $\mathcal{E}'_I$. Indeed, take $Qu:\{y\}\rightarrow\{x\}$, and $\psi =(\psi_0,\psi_1):Q\rightarrow N$, such that $\psi_0$ and $\psi_1$ are the inclusions, then the pullback $\psi^*(\varphi)$ of $\varphi$ along $\psi$ does not belong to $\mathcal{E}_I$, since the empty set is the domain of $(\psi^*(\varphi))_0:\emptyset\rightarrow\{y\}$.\end{examples}

\section{The case of right Kan extensions for presheaves}\label{sec:right Kan extensions}

The right Kan extensions adjunction $Ran_K\vdash Set^K :Set^\mathbb{A}\rightarrow Set^\mathbb{B}$ is more interesting for the purpose of obtaining \emph{very-well-behaved} reflections, than the left Kan extensions case of the precedent section \ref{sec:left Kan extensions}.

As $Set^K$ is left-exact ``a priori" (in fact, it preserves both limits and colimits), Corollary \ref{corollary:ob(N) generating set for C and edm} (of Theorem \ref{theorem:ob(N) generating set for C}; cf.\ also Examples \ref{example:Barr exact cats}) states that $H\vdash I:Set^\mathbb{A}\rightarrow Mono(Set^K)$ is \emph{very-well-behaved} if $Set^K\circ y$ is faithful ($\Leftrightarrow \forall_{A\in\mathbb{A}}\ \mathbb{A}(A,-)\in Mono(Set^K)$) and injective on objects, where $y:\mathbb{A}^{op}\rightarrow Set^\mathbb{A}$ is the Yoneda embedding. But, for the case $Ran_K\vdash Set^K :Set^\mathbb{A}\rightarrow Set^\mathbb{B}$, the injectivity on objects of $Set^K\circ y$ is not needed, because coproduts preserve jointly monic morphisms in $Set$ (see Definition \ref{def:copresjointlymonic} and Lemma \ref{lemma:RightKanCoprodUnitsMonic}).

Next Theorem \ref{theorem:RightKanCogeneratingVeryWellBehaved} gives a useful characterization of such a case using the notion of \emph{cogenerating set} (the dual of the notion of a generating set given two paragraphs before Lemma \ref{lemma:ob(N) generating set for C}).

\begin{lemma}\label{lemma:RightKanUnit}
Consider an adjunction $$(\mathbb{S}^K,Ran_K,\theta,\varepsilon):\mathbb{S}^\mathbb{A}\rightarrow \mathbb{S}^\mathbb{B},$$ given by right Kan extensions calculated as pointwise limits. The right Kan extension $Ran_K(SK)$ is calculated for each $A\in\mathbb{A}$ as

$Ran_K(SK)(A)=Lim((A\downarrow K)$\begin{picture}(30,17)(0,0)
\put(3,4){\vector(1,0){22}}\put(10,7){$Q$}
\end{picture}$\mathbb{B}$\begin{picture}(30,17)(0,0)
\put(3,4){\vector(1,0){22}}\put(5,7){$SK$}
\end{picture}$\mathbb{S})=Lim_fSK(B)$,

\noindent $f\in (A\downarrow K)$, where $Q$ is the projection of the comma category (cf.\ \cite[\S X.3]{SM:cat}).\vspace{3pt}

Then, the unit morphism $\theta_S$ is the unique morphism in $\mathbb{S}^\mathbb{A}$ such that $$\lambda^{SK,A}_f\circ\theta_{S,A}=Sf,$$ for every $f\in (A\downarrow K)$, where $\lambda^{SK,A}_f:Ran_K(SK)(A)\rightarrow SK(B)$ is the morphism in the limiting cone of $Ran_K(SK)(A)$ associated with $f:A\rightarrow K(B)$.
\end{lemma}

\begin{proof}$(1)$ It will be shown first that $\theta_S: S\rightarrow Ran_K(SK)$ is a morphism in $\mathbb{S}^\mathbb{A}$, that is, $$Ran_K(SK)g\circ\theta_{S,A}=\theta_{S,A'}\circ Sg,$$ for every $g:A\rightarrow A'$ in $\mathbb{A}$:\vspace{5pt}

$\lambda^{SK,A'}_{f'}\circ Ran_K(SK)g\circ\theta_{S,A}=\lambda^{SK,A}_{f'\circ g}\circ\theta_{S,A}=S(f'\circ g)=Sf'\circ Sg=\lambda^{SK,A'}_{f'}\circ\theta_{S,A'}\circ Sg$, for any $f'\in (A'\downarrow K)$; the conclusion follows from the universality of the limiting cone $(\lambda^{SK,A'}_{f'})_{f'\in (A'\downarrow K)}$.\\

$(2)$ Secondly, it is going to be proved that the family $(\theta_S)_{S\in\mathbb{S}^\mathbb{A}}$ is indeed a natural transformation $\theta :1_{\mathbb{S}^\mathbb{A}}\rightarrow Ran_K\circ \mathbb{S}^K$, that is, $$(*)\ (Ran_K\sigma K)_A\circ\theta_{S,A}=\theta_{S',A}\circ\sigma_A,$$ for every morphism $\sigma :S\rightarrow S'$ in $\mathbb{S}^\mathbb{A}$ and every $A\in\mathbb{A}$.

That would be so if $$(**)\ \lambda^{S'K,A}_f\circ (Ran_K\sigma K)_A=\sigma_{K(B)}\circ\lambda^{SK,A}_f,$$ for every $f:A\rightarrow K(B)$, where $\lambda^{SK,A}_f$ and $\lambda^{S'K,A}_f$ are the obvious morphisms in the respective limiting cones of $Ran_K(SK)(A)$ and $Ran_K(S'K)(A)$. Indeed,\vspace{5pt}

$\lambda^{S'K,A}_f\circ\theta_{S',A}\circ\sigma_A=S'f\circ\sigma_A$ (by definition of $\theta_{S',A}$)

$=\sigma_{K(B)}\circ Sf$ (by naturality of $\sigma$)

$=\sigma_{K(B)}\circ\lambda^{SK,A}_f\circ\theta_{S,A}$ (by definition of $\theta_{S,A}$)

$=\lambda^{S'K,A}_f\circ (Ran_K\sigma K)_A\circ\theta_{S,A}$ (by hypothesis $(**)$);

\noindent therefore, by the universality of the limiting cone $(\lambda^{S'K,A}_f)_{f\in (A\downarrow K)}$, and under the hypothesis $(**)$, $\theta$ is a natural transformation.\\

$(2.1)$ We are going to prove $(**)$ first for the case $A=K(B)$ and $f=1_{K(B)}$.

It is known that $\lambda^{SK,K(B)}_{1_{K(B)}}=\varepsilon_{SK,B}$ and $\lambda^{S'K,K(B)}_{1_{K(B)}}=\varepsilon_{S'K,B}$ by the definition of the counit $\varepsilon$ (given in Theorem 1 in \cite[\S X.3]{SM:cat}).

Hence, $(**)$ reduces to $$\varepsilon_{S'K,B}\circ (Ran_K\sigma K)_{K(B)}=\sigma_{K(B)}\circ\varepsilon_{SK,B},$$

\noindent which is true since $\varepsilon_{S'K}\cdot ((Ran_K\sigma K)\circ K)=\sigma K\cdot\varepsilon_{SK}$, being $\varepsilon$ the counit of $\mathbb{S}^K\dashv Ran_K$.\\

$(2.2)$ Now,

$\lambda^{S'K,A}_f\circ (Ran_K\sigma K)_A = \lambda^{S'K,K(B)}_{1_{K(B)}}\circ Ran_K(S'K)f\circ(Ran_K\sigma K)_A=\lambda^{S'K,K(B)}_{1_{K(B)}}\circ (Ran_K\sigma K)_{K(B)}\circ Ran_K(SK)f=^{(2.1)}\sigma_{K(B)}\circ\lambda^{SK,K(B)}_{1_{K(B)}}\circ Ran_K(SK)f=\sigma_{K(B)}\circ\lambda^{SK,A}_f$, as wanted.\\

$(3)$ According to Theorem 2(v) in \cite[\S IV.1]{SM:cat}, in order that, as defined in the statement, $\theta$ is the unit of the adjunction, it has to be checked that
\begin{itemize}
\item[(a)] $(\varepsilon\circ\mathbb{S}^K)\cdot(\mathbb{S}^K\circ\theta)=\mathbb{S}^K$, and
\item[(b)] $(Ran_K\circ\varepsilon)\cdot (\theta\circ Ran_K)=Ran_K$,
\end{itemize}
\noindent since $\mathbb{S}^K$ and $Ran_K$ are functors, $\theta$ has just been proved to be a natural transformation, and $\varepsilon$ is the known counit of $\mathbb{S}^K\dashv Ran_K$.\\

$(a)$ For every $S\in\mathbb{S}^\mathbb{A}$ and every $B\in\mathbb{B}$, $$\varepsilon_{SK,B}\circ\theta_{S,K(B)}=1_{SK(B)}$$ follows immediately from the definition of $\theta_S$, since $\varepsilon_{SK,B}=\lambda^{SK,K(B)}_{1_{K(B)}}$ (cf.\ Theorem 1 in \cite [\S X.3]{SM:cat}).\\

$(b)$ At last, we need to show that $Ran_K\varepsilon_T\cdot\theta_{Ran_K(T)}=Ran_K(T)$, for every $T\in\mathbb{S}^\mathbb{B}$:

since $\varepsilon$ is a natural transformation,

$\varepsilon_T\cdot\varepsilon_{Ran_K(T)\circ K}=\varepsilon_T\cdot((Ran_K\varepsilon_T)\circ K)\Rightarrow$

$\Rightarrow \varepsilon_T\cdot\varepsilon_{Ran_K(T)\circ K}\cdot (\theta_{Ran_K(T)}\circ K)=\varepsilon_T\cdot((Ran_K\varepsilon_T)\circ K)\cdot (\theta_{Ran_K(T)}\circ K)$

$\Rightarrow \varepsilon_T\cdot (Ran_K(T)\circ K)=\varepsilon_T\cdot((Ran_K\varepsilon_T)\circ K)\cdot (\theta_{Ran_K(T)}\circ K)$, by $(a)$ with $S=Ran_K(T)$

$\Rightarrow \varepsilon_T\cdot (Ran_K(T)\circ K)=\varepsilon_T\cdot((Ran_K\varepsilon_T\cdot \theta_{Ran_K(T)})\circ K)$

$\Rightarrow Ran_K(T)=Ran_K\varepsilon_T\cdot\theta_{Ran_K(T)}$, since $\varepsilon_T$ is a counit morphism for $\mathbb{S}^K\dashv Ran_K$.\end{proof}

\begin{lemma}\label{lemma:jointlymonicRightKanUnit}
Under the conditions in the previous Lemma \ref{lemma:RightKanUnit}, for each $S\in\mathbb{S}^\mathbb{A}$ and $A\in\mathbb{A}$, $\theta_{S,A}:S(A)\rightarrow Ran_K(SK)(A)$ is a monomorphism in $\mathbb{S}$ if and only if the morphisms in the set $\{Sf:S(A)\rightarrow SK(B)|f:A\rightarrow K(B),B\in\mathbb{B}\}$ are jointly monic in $\mathbb{S}$.
\end{lemma}
\begin{proof}
The easy proof is displayed in \cite[Lemma 4.1]{X:well behaved}. Notice that the just above Lemma \ref{lemma:RightKanUnit} in the present paper completes that proof, by justifying the nature of the unit $\theta$ assumed in \cite[Lemma 4.1]{X:well behaved}.\end{proof}

\begin{definition}\label{def:copresjointlymonic}
Consider, in a category $\mathbb{S}$ with coproducts, a family $(f_{ij}:s_i\rightarrow s_{ij})_{j\in J}$ of morphisms in $\mathbb{S}$ with the same domain $s_i$ which are jointly monic, for each $i\in I$ (notice that $J$ is fixed). It will be said that the coproducts preserve jointly monic morphisms in $\mathbb{S}$ if, for any such family of families, the family $(\coprod_{i\in I}f_{ij}:\coprod_{i\in I}s_i\rightarrow \coprod_{i\in I}s_{ij})_{j\in J}$ is also jointly monic, where $\coprod_{i\in I}f_{ij}$ is the morphism uniquely determined by the coproduct diagram associated to $(f_{ij})_{i\in I}$.

As an example, it is trivial that coproducts preserve jointly monic morphisms in the category $Set$ of sets.
\end{definition}

\begin{lemma}\label{lemma:RightKanCoprodUnitsMonic}
Under the conditions of Lemmas \ref{lemma:RightKanUnit}, suppose that $\mathbb{S}$ has coproducts which preserve jointly monic morphisms (cf.\ previous Definition \ref{def:copresjointlymonic} just above).

If $(S_i)_{i\in I}$ is a family of functors belonging to $\mathbb{S}^\mathbb{A}$, such that $\theta_{S_i}:S_i\rightarrow Ran_K(S_iK)$ is monic, for all $i\in I$, then the unit $\theta_{\coprod_{i\in I}S_i}$ of the coproduct is also monic.
\end{lemma}

\begin{proof}
We need to show that $\theta_{\coprod_{i\in I}S_i,A}:\coprod_{i\in I}S_i(A)\rightarrow Ran_K(\coprod_{i\in I}S_iK)(A)$ is a monomorphism for every $A\in\mathbb{A}$; or, equivalently, that the morphisms in the sets $$\{\coprod_{i\in I}S_if:\coprod_{i\in I}S_i(A)\rightarrow \coprod_{i\in I}S_iK(B)\ |\ f:A\rightarrow K(B),B\in\mathbb{B}\}$$ \noindent are jointly monic for every $A\in \mathbb{A}$, according to Lemma \ref{lemma:jointlymonicRightKanUnit}.

Since $\theta_{S_i,A}:S_i(A)\rightarrow Ran_K(S_iK)(A)$ is a monomorphism by assumption, for all $i\in I$ and $A\in \mathbb{A}$, then all the sets $\{S_if:S_i(A)\rightarrow S_iK(B)|f:A\rightarrow K(B),B\in\mathbb{B}\}$ are jointly monic, again by Lemma \ref{lemma:jointlymonicRightKanUnit}. Hence, as coproducts preserve jointly monic morphisms in $\mathbb{S}$, it follows that $\theta_{\coprod_{i\in I}S_i}$ is monic.\end{proof}

\begin{theorem}\label{theorem:RightKanCogeneratingVeryWellBehaved}

Let $K:\mathbb{B}\rightarrow\mathbb{A}$ be a functor from a small category $\mathbb{B}$ into a category $\mathbb{A}$ with small hom-sets. There is an adjunction $Ran_K\vdash Set^K:Set^\mathbb{A}\rightarrow Set^\mathbb{B}$ given by right Kan extensions calculated as pointwise limits. Then, factorizing its unit morphisms $\theta_S:S\rightarrow Ran_K(SK)$ componentwise with the surjections and injections, gives rise to the full reflection with stable units $H\vdash I:Set^\mathbb{A}\rightarrow Mono(Set^K)$.

If $K(ob\mathbb{B})$ is a cogenerating set for $\mathbb{A}$, then $(\mathcal{E}'_I,\mathcal{M}^*_I)$ is a monotone-light factorization system on $Set^\mathbb{A}$ (where $K(ob\mathbb{B})$ is the image by $K$ of every object in $\mathbb{B}$).
\end{theorem}

\begin{proof}
Consider two morphisms $f,g:A'\rightarrow A$ in $\mathbb{A}$, and their image $f^*,g^*:\mathbb{A}(A,K(-))\rightarrow \mathbb{A}(A',K(-))$ by $Set^K\circ y$. These two morphisms $f^*$ and $g^*$ are distinct functors if and only if there exists $B\in\mathbb{B}$ and a morphism $h:A\rightarrow K(B)$ such that $f^*_B(h)=h\circ f\neq h\circ g=g^*_B(h)$ (see \cite{X:concordant&monotone}, just before Theorem 6.1 there, where this part of the proof was given).

In the last paragraph, it has been shown that $Set^K\circ y$ is faithful if and only if $K(ob\mathbb{B})$ is a cogenerating set for $\mathbb{A}$. It then follows, by Lemma \ref{lemma:ob(N) generating set for C}$(b)$ that all representable functors $\mathbb{A}(A,-)$ are in $Mono(Set^K)$, $A\in\mathbb{A}$, if and only if $K(ob\mathbb{B})$ is a cogenerating set for $\mathbb{A}$. Then, by the precedent Lemma \ref{lemma:RightKanCoprodUnitsMonic}, with $\mathbb{S}=Set$, all the coproducts $\coprod_{i\in I}\mathbb{A}(A_i,-)$ are in $Mono(Set^K)$. Hence, the conditions of Theorem \ref{theorem:stable units+enough edm} hold (cf.\ Examples \ref{example:Barr exact cats}), and so there is a monotone-light factorization system derived from $Set^K\dashv Ran_K$ if $K(ob\mathbb{B})$ is a cogenerating set for $\mathbb{A}$.\end{proof}

\begin{examples}\label{example:m-l fact simplicial sets}
Consider the category $\mathbf{\Delta}$ of non-empty finite ordinal numbers $[n]$ ($n\geq 0$), where the singular set $\{[0]\}$ is a generating set ($\mathbf{\Delta}=\mathbf{\Delta}^+$ as in \cite[VII.5]{SM:cat}: $[n]$ is the linear order whose objects are $0,1,...,n$; the morphisms of $\mathbf{\Delta}$ are the weakly monotone functions). Hence, $\{[0]\}$ is a cogenerating set in $\mathbf{\Delta}^{op}$. Let $K:\mathbb{B}\rightarrow \mathbf{\Delta}^{op}$ be a functor (from a small category $\mathbb{B}$) such that $[0]$ is in the image $K(ob\mathbb{B})$ of the objects of $\mathbb{B}$. Then, $H\vdash I:Set^{\mathbf{\Delta}^{op}}\rightarrow Mono(Set^K)$ is a very-well-behaved reflection.

The same conclusion holds if one replaces $\mathbf{\Delta}$ by its full subcategory $\mathbf{\Delta}_n$, determined by the $n+1$ objects $[0]$, $[1]$, ..., $[n]$.

For instance, if $\mathbb{B}=\mathbf{\Delta}^{op}_0\cong \mathbf{1}$, with $K:\mathbf{\Delta}^{op}_0\rightarrow\mathbf{\Delta}^{op}$ the full inclusion, one gets the non-trivial (i.e.\ $\mathcal{E}'_I$ strictly contained in $\mathcal{E}_I$) monotone-light factorization for simplicial sets via ordered simplicial complexes (cf.\ \cite{X:simplicialsets}).
\end{examples}

\section{General Results for Sub-reflections from Models}\label{sec:sub-reflectionsfromodels}

In the present paper, a \emph{sketch} consists of a category $\mathbb{A}$, together with a chosen set of distinguished cones of the form $$\lambda:\Delta (a)\rightarrow L,$$ where $\Delta :\mathbb{\mathbb{A}}\rightarrow \mathbb{A}^\mathbb{J}$ is the diagonal functor, $\mathbb{J}$ is a small category and $L:\mathbb{J}\rightarrow\mathbb{A}$ is a functor ($\mathbb{J}$, $L$, $a$, $\lambda$ vary from one distinguished cone to another).

$\check{\mathbb{A}}(\mathbb{S})$ shall denote the full subcategory of the functor category $\hat{\mathbb{A}}(\mathbb{S})=\mathbb{S}^\mathbb{A}$ determined by the models of such a sketch, that is, whose objects are the functors $M:\mathbb{A}\rightarrow \mathbb{S}$ for which $M\lambda =M\Delta (a)\rightarrow ML$ is a limiting cone in $\mathbb{S}$, for each distinguished cone in the sketch.

If $\mathbb{S}=Set$ then we will simply write $\hat{\mathbb{A}}$ ($=\hat{\mathbb{A}}(Set)=Set^\mathbb{A}$) and $\check{\mathbb{A}}$ ($=\check{\mathbb{A}}(Set)$).

If there is a factorization system $(\mathcal{E},\mathcal{M})$ on $\mathbb{S}$, then $(\mathcal{\hat{E}},\mathcal{\hat{M}})=(\mathcal{E}^\mathbb{A},\mathcal{M}^\mathbb{A})$ will denote the factorization system which is the uplifting of $(\mathcal{E},\mathcal{M})$ to $\hat{\mathbb{A}}(\mathbb{S})$ (cf.\ Lemma \ref{lemma:uplifting (pre)factsys}$(b)$).

\begin{lemma}\label{lemma:catModClosedUnderFactFinLim}

Let $(\mathcal{E},\mathcal{M})$ be a stable factorization system $(\mathcal{E}=\mathcal{E}')$ on a category $\mathbb{S}$ with pullbacks.

Consider any natural transformation $\varphi:F\rightarrow G:\mathbb{J}\rightarrow\mathbb{S}$, where $\mathbb{J}$ is a finite category and both functors $F$, $G$ are supposed to have limits in $\mathbb{S}$, with cones $(LimF,(\lambda^F_i)_{i\in\mathbb{J}})$ and $(LimG,(\lambda^G_i)_{i\in\mathbb{J}})$.

Consider also the functor $D:\mathbb{J}\rightarrow\mathbb{S}$ obtained by factorizing $\varphi =n\cdot f:F\rightarrow D\rightarrow G$ using the factorization system $(\hat{\mathcal{E}},\hat{\mathcal{M}})$, which is the uplifting of the factorization system $(\mathcal{E},\mathcal{M})$ to $\hat{\mathbb{S}}(\mathbb{J})=\mathbb{S}^\mathbb{J}$ (i.e., using the $(\mathcal{E},\mathcal{M})$-factorization for each component $\varphi_i=n_i\circ f_i:F(i)\rightarrow D(i)\rightarrow G(i)$, $i\in\mathbb{J}$).

Then:
\begin{itemize}
\item[$(a)$]
if $\mathbb{J}=\begin{picture}(55,15)(0,5)

\put (0,4){$i$}\put (25,4){$k$} \put (50,4){$j$}

\put(10,7){\vector(1,0){10}}\put(10,10){$u$}
\put(45,7){\vector(-1,0){10}}\put(37,10){$v$}

\end{picture}$, being $LimF$ and $LimG$ pullbacks, then the pullback $LimD$ is isomorphic to the object $s$ obtained by the $(\mathcal{E},\mathcal{M})$-factorization of the canonical morphism $\varphi_i\times\varphi_j=\mu\circ\eta :LimF\rightarrow s\rightarrow LimG$;
\item[$(b)$]
if $\mathbb{\mathbb{J}}$ is the discrete category generated by the finite set $\{1,2,...,n\}$, ($n\geq 0$), being $LimF=\prod^n_{i=1}F(i)$ and $LimG=\prod^n_{i=1}G(i)$ $n$-products, then the $n$-product $LimD=\prod^n_{i=1}D(i)$ is isomorphic to the object $s$ obtained by the $(\mathcal{E},\mathcal{M})$-factorization of the canonical morphism $\prod^n_{i=1}\varphi_i =\mu\circ\eta :\prod^n_{i=1}F(i)\rightarrow\prod^n_{i=1}G(i)$;
\item[$(c)$]
if $\mathbb{S}$ has furthermore binary products and $\mathbb{J}=\ \downarrow\downarrow$ ($=\begin{picture}(30,13)(0,5)

\put (0,4){$0$}\put (25,4){$1$}

\put(10,9){\vector(1,0){10}}
\put(10,6){\vector(1,0){10}}

\end{picture}$), being $LimF$ and $LimG$ equalizers, then the equalizer $LimD$ is isomorphic to the object $s$ obtained by the $(\mathcal{E},\mathcal{M})$-factorization of the canonical morphism $Lim\varphi =\mu\circ\eta :LimF\rightarrow s\rightarrow LimG$;
\item[$(d)$]
if, furthermore, $\mathbb{S}$ is finitely-complete (i.e., it has also a terminal object $t$), then $LimD$ is isomorphic to the object $s$ obtained by the $(\mathcal{E},\mathcal{M})$-factorization of the canonical morphism $Lim\varphi =\mu\circ\eta :LimF\rightarrow s\rightarrow LimG$.
\end{itemize}
\end{lemma}

\begin{proof}
$(a)$ Consider the commutative diagram

\begin{picture}(300,180)
\put(15,150){$LimF$}\put(280,150){$F(j)$}
\put(20,0){$F(i)$}
\put(50,153){\vector(1,0){225}}\put(160,160){$\lambda^F_j$}
\put(30,145){\vector(0,-1){133}}\put(35,70){$\lambda^F_i$}
\put(50,3){\vector(1,0){45}}\put(60,8){$f_i$}
\put(295,145){\vector(0,-1){33}}\put(298,125){$f_j$}
\put(55,145){\vector(1,-1){38}}\put(78,125){$f_i\times f_j$}

\put(110,0){$D(i)$}\put(105,100){$LimD$}
\put(280,100){$D(j)$}
\put(125,50){$\lambda^D_i$}\put(200,110){$\lambda^D_j$}
\put(140,103){\vector(1,0){135}}\put(120,95){\vector(0,-1){82}}\put(140,3){\vector(1,0){45}}
\put(150,8){$n_i$}
\put(295,95){\vector(0,-1){33}}\put(298,75){$n_j$}
\put(135,95){\vector(1,-1){38}}\put(157,75){$n_i\times n_j$}

\put(190,0){$G(i)$}\put(185,50){$LimG$}
\put(280,0){$G(k)$\hspace{10pt,}}\put(280,50){$G(j)$}
\put(202,25){$\lambda^G_i$}\put(298,25){$Gv$}
\put(240,8){$Gu$}\put(240,58){$\lambda^G_j$}
\put(220,3){\vector(1,0){55}}\put(220,53){\vector(1,0){55}}
\put(200,45){\vector(0,-1){33}}\put(295,45){\vector(0,-1){33}}
\end{picture}
\vspace{10pt}

\noindent where $LimD$ is the pullback of \begin{picture}(130,15)(0,5)
\put (0,4){$D(i)$}\put (50,4){$D(k)$} \put (100,4){$D(j)$.}

\put(26,7){\vector(1,0){20}}\put(26,10){$Du$}
\put(98,7){\vector(-1,0){20}}\put(79,10){$Dv$}
\end{picture}

According to the properties of a factorization system (cf.\ \cite[Proposition 2.2]{CJKP:stab}), one needs to show that $n_i\times n_j\in \mathcal{M}$ and $f_i\times f_j\in\mathcal{E}$ in order that $s\cong LimD$. That is obviously so, since $\mathcal{M}$ is closed under pullbacks in any factorization system, and also is $\mathcal{E}$ by hypothesis ($\mathcal{E}=\mathcal{E}'$).\\

$(b)$ If $n=0$ and hence $F=\emptyset =G$, then the $(\mathcal{E},\mathcal{M})$-factorization of the canonical morphism is $t=t=t$, where $t=Lim\emptyset =LimF=LimG$ is a terminal object of $\mathbb{S}$.

If $n=1$ it is also obvious, since $LimF=F(1)$, $LimG=G(1)$ and $LimD=D(1)$: $\varphi_1=n_1\circ f_1:F(1)\rightarrow D(1)\rightarrow G(1).$

If $n=2$, then $LimF=F(1)\times F(2)$, $LimG=G(1)\times G(2)$ are binary products. Consider the commutative diagram

\begin{picture}(370,110)

\put(0,80){$D(1)$}\put(90,83){\vector(-1,0){60}}
\put(57,89){$\lambda^D_1$}
\put(95,80){$D(1)\times D(2)$}\put(163,83){\vector(1,0){60}}
\put(182,89){$\lambda^D_2$}\put(230,80){$D(2)$}

\put(15,73){\vector(0,-1){60}}\put(0,40){$n_1$}
\put(115,40){$n_1\times n_2$}\put(127,73){\vector(0,-1){60}}
\put(250,40){$n_2$}\put(245,73){\vector(0,-1){60}}

\put(60,45){$P(1)$}\put(85,40){\vector(1,-1){25}}\put(80,25){$p^1_1$}
\put(60,60){\vector(-2,1){30}}\put(50,70){$p^1_2$}
\put(175,45){$P(2)$}\put(170,40){\vector(-1,-1){25}}\put(165,25){$p^2_1$}
\put(200,60){\vector(2,1){30}}\put(200,70){$p^2_2$}
\put(110,75){\vector(-1,-1){20}}\put(85,65){$q_1$}
\put(150,75){\vector(1,-1){20}}\put(165,65){$q_2$}

\put(0,0){$G(1)$}\put(90,3){\vector(-1,0){60}}
\put(57,9){$\lambda^G_1$}
\put(95,0){$G(1)\times G(2)$}\put(163,3){\vector(1,0){60}}
\put(182,9){$\lambda^G_2$}\put(230,0){$G(2)$,}
\end{picture}\vspace{5pt}

\noindent where the three squares $\lambda^G_1\circ p^1_1=n_1\circ p^1_2$, $\lambda^G_2\circ p^2_1=n_2\circ p^2_2$ and $p^1_1\circ q_1=p^2_1\circ q_2$ are pullback diagrams, $p^i_2\circ q_i=\lambda^D_i$ and $n_1\times n_2=p^i_1\circ q_i$ ($i=1,2$). It was just shown that $LimD=D(1)\times D(2)$ exists and can be obtained using pullbacks, which are supposed to exist in the category $\mathbb{S}$.\\

Consider now the next commutative diagram

\begin{picture}(370,110)

\put(0,80){$F(1)$}\put(90,83){\vector(-1,0){60}}
\put(57,89){$\lambda^F_1$}
\put(95,80){$F(1)\times F(2)$}\put(163,83){\vector(1,0){60}}
\put(182,89){$\lambda^F_2$}\put(230,80){$F(2)$}

\put(15,73){\vector(0,-1){20}}\put(0,60){$f_1$}
\put(85,60){$f_1\times f_2$}\put(127,73){\vector(0,-1){20}}
\put(250,60){$f_2$}\put(245,73){\vector(0,-1){20}}

\put(0,40){$D(1)$}\put(90,43){\vector(-1,0){60}}
\put(57,49){$\lambda^D_1$}
\put(95,40){$D(1)\times D(2)$}\put(163,43){\vector(1,0){60}}
\put(182,49){$\lambda^D_2$}\put(230,40){$D(2)$}

\put(15,33){\vector(0,-1){20}}\put(0,20){$n_1$}
\put(85,20){$n_1\times n_2$}\put(127,33){\vector(0,-1){20}}
\put(250,20){$n_2$}\put(245,33){\vector(0,-1){20}}

\put(0,0){$G(1)$}\put(90,3){\vector(-1,0){60}}
\put(57,9){$\lambda^G_1$}
\put(95,0){$G(1)\times G(2)$}\put(163,3){\vector(1,0){60}}
\put(182,9){$\lambda^G_2$}\put(230,0){$G(2)$,}

\end{picture}\vspace{5pt}

\noindent where $n_1\times n_2\in\mathcal{M}$ (cf.\ Proposition 2.2(d) in \cite{CJKP:stab}). Then, one has just to show that $f_1\times f_2\in\mathcal{E}$, which is so since the binary product $F(1)\times F(2)$ may be obtained via pullbacks (from the product diagram of $D(1)\times D(2)$, exactly in the same way as we showed that $D(1)\times D(2)$ exists using the product diagram of $G(1)\times G(2)$), and the class of morphisms $\mathcal{E}$ of $\mathbb{S}$ is stable under pullbacks ($\mathcal{E}=\mathcal{E}'$) 







For $n\geq 3$, the statement follows immediately by iteration of binary products.\\

$(c)$ Proposition 2.2(d) in \cite{CJKP:stab} confirms that $\mu\in\mathcal{M}$. In order to prove that $\eta\in\mathcal{E}'$ one needs only to invoke items $(a)$ and $(b)$, already proved, and the canonical presentation of an equalizer diagram using binary products and pullbacks.\\

$(d)$ Proposition 2.2(d) in \cite{CJKP:stab} confirms that $\mu\in\mathcal{M}$. In order to prove that $\eta\in\mathcal{E}'$ one needs only to invoke items $(a)$, $(b)$ and $(c)$, already proved, and the presentation of $LimF$ and $LimD$ as equalizers of two canonical morphisms between (finite) products.\end{proof}

\begin{proposition}\label{proposition:stableunitsModelsPresheaves}
Consider an adjunction $Ran_K\vdash Set^K:\hat{\mathbb{A}}=Set^{\mathbb{A}}\rightarrow \hat{\mathbb{B}}=Set^\mathbb{B}$, as in Theorem \ref{theorem:RightKanCogeneratingVeryWellBehaved}, such that $K(ob\mathbb{B})$ is a cogenerating set for $\mathbb{A}$.

Let $\check{\mathbb{A}}$ be the full subcategory determined by the models of a (fixed) sketch with every $\mathbb{J}$ finite (cf.\ the beginning of this section \ref{sec:sub-reflectionsfromodels}).

Suppose also that $Ran_K\circ Set^K(M)\in \check{\mathbb{A}}$ for every $M\in\check{\mathbb{A}}$.

Then, there is a reflection with stable units $\check{H}\vdash \check{I}:\check{\mathbb{A}}\rightarrow\check{\mathbb{M}}$ which is a subreflection of $H\vdash I:\hat{\mathbb{A}}\rightarrow\mathbb{M}=Mono(Set^K)$, and $(\mathcal{E}_{\check{I}},\mathcal{M}_{\check{I}})=(\mathcal{E}_I\bigcap Mor(\check{\mathbb{A}}),\mathcal{M}_I\bigcap Mor(\check{\mathbb{A}}))$. Here, $\check{\mathbb{M}}$ is the full subcategory of $\mathbb{M}$ determined by the models of the given sketch on $\hat{\mathbb{A}}$, and $(\mathcal{E}_{\check{I}},\mathcal{M}_{\check{I}})$, $(\mathcal{E}_I,\mathcal{M}_I)$ are the reflective factorization systems associated respectively to the reflections $\check{H}\vdash \check{I}$ and $H\vdash I$.
\end{proposition}

\begin{proof}
As, by hypothesis, for every $M\in\check{\mathbb{M}}$, $Ran_K(MK)\in \check{\mathbb{M}}$, then $I(M)\in\check{\mathbb{M}}$ according to Lemma \ref{lemma:catModClosedUnderFactFinLim}$(d)$ just above: consider, in the statement of Lemma \ref{lemma:catModClosedUnderFactFinLim}, $F=M\circ L, G=Ran_K(MK)\circ L:\mathbb{J}\rightarrow \mathbb{A}\rightarrow Set=\mathbb{S}$, $\varphi=\theta_M\circ L$. Therefore, $\check{\mathbb{M}}$ is indeed a full reflective subcategory of $\check{\mathbb{A}}$.

Let $(\mathcal{E}_{\check{I}},\mathcal{M}_{\check{I})}$ be the reflective factorization system associated to $\check{I}\dashv \check{H}$. It is known that $\gamma: M\rightarrow N$ is in $\mathcal{E}_{\check{I}}$ if and only if $\check{I}\gamma =I\gamma$ is an isomorphism (cf.\ \cite[\S 3.1]{CJKP:stab}), therefore $\mathcal{E}_{\check{I}}=\mathcal{E}_I\bigcap Mor(\check{\mathbb{A}})$. As $\check{\mathbb{A}}$ is closed under limits in $\hat{\mathbb{A}}$ (using the property of interchange of limits, cf.\ \cite[\S IX.2]{SM:cat}), it follows that $\mathcal{E}'_I\bigcap Mor(\check{\mathbb{A}})\subseteq \mathcal{E}'_{\check{I}}$. Hence, $\check{I}\dashv \check{H}$ has stable units because $I\dashv H$ has stable units (the units $\check{\eta}_M=\eta_M\in \mathcal{E}'_I$, $M\in\check{\mathbb{A}}$). Being $\check{H}\vdash \check{I}$ a reflection with stable units it is also simple, and then $\mathcal{M}_{\check{I}}=\mathcal{M}_I\bigcap Mor(\check{\mathbb{A}})$, because $\check{\mathbb{A}}$ is closed under limits in $\hat{\mathbb{A}}$ and by the characterization of the morphisms in $\check{\mathbb{M}}$ in a simple reflection (cf.\ Theorem 4.1 in \cite{CHK:fact} or \cite[\S 3.5]{CJKP:stab}).\end{proof}

\begin{examples}\label{example:m-l fact catsviapreord}
Consider $\mathbb{P}=\mathbf{\Delta}^{op}_3$ the dual of the truncated simplicial category generated by the diagram (where $0\leq i\leq k+1$ for $d^k_i$ and $0\leq j\leq k$ for $s^k_j$, $k\in\{0,1,2\}$)\\

\begin{picture}(200,25)(0,0)
\put(2,0){$[3]$}
\put(20,10){\vector(1,0){40}}\put(60,-5){\vector(-1,0){40}}
\put(33,-2){$s^2_j$}\put(33,15){$d^2_i$}
\put(63,0){$[2]$}
\put(80,10){\vector(1,0){40}}\put(120,-5){\vector(-1,0){40}}
\put(95,-2){$s^1_j$}\put(95,15){$d^1_i$}
\put(125,0){$[1]$}
\put(140,10){\vector(1,0){40}}\put(180,-5){\vector(-1,0){40}}
\put(155,-2){$s^0_0$}\put(155,15){$d^0_i$}
\put(185,0){$[0]$}
\end{picture}\\

\noindent and the usual relations (cf.\ Examples \ref{example:m-l fact simplicial sets} and \cite[\S VII.5]{SM:cat}; notice that this example could as well be worked out using the dual of the entire simplicial category $\mathbb{A}=\mathbf{\Delta}^{op}$)\footnote{It follows from the Lemma in \cite[\S VII.5]{SM:cat} that: truncating the dual of the simplicial category gives the same result as using the diagram above and the usual relations to generate $\mathbf{\Delta}^{op}_3$.}.

Let $K:\mathbf{\Delta}^{op}_0\rightarrow\mathbf{\Delta}^{op}_3=\mathbb{P}$ be the inclusion functor ($\mathbf{\Delta}^{op}_0\cong\mathbf{1}$). Since $\{[0]\}$ is a cogenerating set for $\mathbf{\Delta}^{op}_3$ (because $[0]$ is a generator in $\mathbf{\Delta}$), then by Theorem \ref{theorem:RightKanCogeneratingVeryWellBehaved} there is a reflection $H\vdash I: \hat{\mathbb{P}}=Set^\mathbb{P}\rightarrow \mathbb{M}=Monos(Set^K)$ which is very-well-behaved (i.e., with stable units and a monotone-light factorization). $\hat{\mathbb{P}}$ could be notated $Smp_3$, the category of truncated simplicial sets.

This reflection induces a subreflection with stable units $\check{H}\vdash \check{I}:\check{\mathbb{P}}\rightarrow \check{\mathbb{M}}$, for the sketch whose models $M:\mathbb{P}\rightarrow Set$ are those functors such that $Md^0_1\circ Md^1_0=Md^0_0\circ Md^1_2$ and $Md^1_2\circ Md^2_0=Md^1_0\circ Md^2_3$ are pullback squares, according to Proposition \ref{proposition:stableunitsModelsPresheaves}. Indeed, $Ran_K\circ Set^K(S)\in\check{\mathbb{P}}$ for every $S\in\hat{\mathbb{P}}$ (not only for every $S\in\check{\mathbb{P}}$): it is easy to check that
$Ran_K(SK)([n])=Lim(([n]\downarrow K)$\begin{picture}(30,17)(0,0)
\put(3,4){\vector(1,0){22}}\put(10,7){$Q$}
\end{picture}$\mathbf{\Delta}^{op}_0$\begin{picture}(30,17)(0,0)
\put(3,4){\vector(1,0){22}}\put(5,7){$SK$}
\end{picture}$Set)=S^{n+1}_0$, the $n+1$-power of $S_0=S([0])$, $0\leq n\leq 3$; hence, $Ran_K(SK)$ is an indiscrete category (a connected equivalence relation). $\check{\mathbb{P}}$ can be identified with $Cat$, the category of all small categories and $\check{\mathbb{M}}$ with the category $Preord$ of all preordered sets.

The reflection $Cat\rightarrow Preord$ was showed to be very-well-behaved in \cite{X:ml}, where the monotone-light factorization was presented and showed to be non-trivial ($\mathcal{E}'_{\check{I}}\neq\mathcal{E}_{\check{I}}$).

The existence of a monotone-light factorization is a consequence of Theorem \ref{theorem:stable units+enough edm}: for every category $M\in\check{\mathbb{P}}$, there is an effective descent morphism $p:E\rightarrow M$ in $\hat{\mathbb{P}}$ which is also an effective descent morphism in $\check{\mathbb{P}}=Cat$ (cf.\ the characterization of effective descent morphisms in $Cat$ given in \cite{JST:edm}); $p$ is the canonical presentation of the presheaf $M$, where $E$ is the coproduct of representable functors which are linear orders, and so $E$ is obviously also a linear order (a special case of a preorder).\end{examples}

\subsection{Discussion of the very-well-behaved reflection from 2-categories into 2-preorders}\label{subsec:2-categories into 2-preorders}

This and the following last subsection are presented in a lighter way than the rest of the paper. We expect to be more precise in a future paper.\\

The example discussed in this subsection was fully studied in \cite{X:2ml}, with a slightly different presentation from the one given here. In order to relate to that paper \cite{X:2ml}, so that the interested reader may have an easier task consulting it, we shall call \emph{precategories} the objects of $\hat{\mathbb{P}}=Set^\mathbb{P}$ in Examples \ref{example:m-l fact catsviapreord} just above, and the diagram there which generates the category $\mathbb{P}$ is now called a \emph{precategory diagram}, with $P_i=[i]$ ($0\leq i\leq 3$).

Consider the category $2\mathbb{P}$ generated by the following \emph{2-precategory} diagram,


\begin{picture}(400,160)

\put(0,140){$P_{43}$}
\put(35,155){\vector(1,0){70}}\put(105,140){\vector(-1,0){70}}

\put(128,140){$P_{42}$}
\put(165,155){\vector(1,0){70}}\put(235,140){\vector(-1,0){70}}

\put(250,140){$P_{41}$}
\put(285,155){\vector(1,0){70}}\put(355,140){\vector(-1,0){70}}
\put(370,140){$P_0$}

\put (0,130){\vector(0,-1){40}}\put (15,90){\vector(0,1){40}}

\put (130,130){\vector(0,-1){40}}\put (145,90){\vector(0,1){40}}

\put (250,130){\vector(0,-1){40}}\put (265,90){\vector(0,1){40}}

\put (358,105){$1_{P_0}$}\put (375,130){\vector(0,-1){40}}

\put(0,70){$P_{33}$}
\put(35,85){\vector(1,0){70}}\put(105,70){\vector(-1,0){70}}

\put(128,70){$P_{32}$}
\put(165,85){\vector(1,0){70}}\put(235,70){\vector(-1,0){70}}

\put(250,70){$P_{31}$}
\put(285,85){\vector(1,0){70}}\put(355,70){\vector(-1,0){70}}
\put(370,70){$P_0$}

\put (0,60){\vector(0,-1){40}}\put (15,20){\vector(0,1){40}}

\put (130,60){\vector(0,-1){40}}\put (145,20){\vector(0,1){40}}

\put (250,60){\vector(0,-1){40}}\put (265,20){\vector(0,1){40}}

\put (358,35){$1_{P_0}$}\put (375,60){\vector(0,-1){40}}

\put(0,0){$P_{23}$}
\put(35,15){\vector(1,0){70}}\put(105,0){\vector(-1,0){70}}

\put(128,0){$P_{22}$}
\put(165,15){\vector(1,0){70}}\put(235,0){\vector(-1,0){70}}

\put(250,0){$P_{21}$}
\put(285,15){\vector(1,0){70}}\put(355,0){\vector(-1,0){70}}
\put(370,0){$P_0$}

\put (0,-10){\vector(0,-1){40}}\put (15,-50){\vector(0,1){40}}

\put (130,-10){\vector(0,-1){40}}\put (145,-50){\vector(0,1){40}}

\put (250,-10){\vector(0,-1){40}}\put (265,-50){\vector(0,1){40}}

\put (358,-35){$1_{P_0}$}\put (375,-10){\vector(0,-1){40}}

\put(0,-70){$P_3$}
\put(35,-55){\vector(1,0){70}}\put(105,-70){\vector(-1,0){70}}
\put(60,-67){$s^2_j$}\put(60,-50){$d^2_i$}
\put(128,-70){$P_2$}
\put(165,-55){\vector(1,0){70}}\put(235,-70){\vector(-1,0){70}}
\put(190,-67){$s^1_j$}\put(190,-50){$d^1_i$}
\put(250,-70){$P_1$}
\put(285,-55){\vector(1,0){70}}\put(355,-70){\vector(-1,0){70}}
\put(190,-67){$s^1_j$}\put(190,-50){$d^1_i$}
\put(370,-70){$P_0$\hspace{10pt,}}
\put(320,-67){$s^0_0$}\put(320,-50){$d^0_i$}

\end{picture}

\vspace{80pt}
\noindent in which:

$\bullet$ each one of the four horizontal diagrams is a precategory diagram;

$\bullet$ each one of the four vertical diagrams is a precategory diagram (the rightmost one is the trivial one with all morphisms identities);

$\bullet$ in the obvious sense given in \cite{X:2ml}, the $s_j$ provide $3$ precategory morphisms vertically and $3$ precategory morphisms horizontally; the same for the $d_i$.\\

Let $K:\mathbb{P}\rightarrow 2\mathbb{P}$ be the inclusion functor of $\mathbb{P}$ in $2\mathbb{P}$

(recall, from Examples \ref{example:m-l fact catsviapreord}, that $\mathbb{P}$ is generated by

\begin{picture}(200,27)(0,0)
\put(2,0){$P_3$}
\put(20,10){\vector(1,0){40}}\put(60,-5){\vector(-1,0){40}}
\put(33,-2){$s^2_j$}\put(33,15){$d^2_i$}
\put(63,0){$P_2$}
\put(80,10){\vector(1,0){40}}\put(120,-5){\vector(-1,0){40}}
\put(95,-2){$s^1_j$}\put(95,15){$d^1_i$}
\put(125,0){$P_1$}
\put(140,10){\vector(1,0){40}}\put(180,-5){\vector(-1,0){40}}
\put(155,-2){$s^0_0$}\put(155,15){$d^0_i$}
\put(185,0){$P_0$}
\end{picture} and the usual relations).\\

$\{P_3=P_{13},P_2=P_{12},P_1=P_{11},P_0=P_{10}\}$ is a cogenerating set for $2\mathbb{P}$, because $P_0$ is a generating object for $2\mathbb{P}^{op}$. Then, by Theorem \ref{theorem:RightKanCogeneratingVeryWellBehaved}, there is a very-well-behaved reflection $$H\vdash I:\hat{2\mathbb{P}}=Set^{\hat{2\mathbb{P}}}\rightarrow 2\mathbb{M}=Monos(Set^K).$$

This reflection induces a subreflection having stable units $$\check{H}\vdash \check{I}:\check{2\mathbb{P}}\rightarrow \check{\mathbb{M}},$$ for the sketch whose models are those functors $M$ such that $Md_1\circ Md_0=Md_o\circ Md_2$ and $Md_2\circ Md_0=Md_o\circ Md_3$ are pullback squares for every horizontal and vertical precategory diagram. Indeed, $Ran_K\circ Set^K(M)\in \check{2\mathbb{P}}$ for every $M\in \check{2\mathbb{P}}$. Being $K$ a full inclusion, it is well known that $Ran_K\vdash Set^K$ is a full reflection (cf.\ Corollary 3 in \cite[\S X.3]{SM:cat}), hence $Ran_K$ can be seen as the inclusion of the category of ``indiscrete 2-graphs" in $\hat{2\mathbb{P}}$. An ``indiscrete 2-graph" $T:2\mathbb{P}\rightarrow Set$ is the presheaf such that $T_{21}=T_1\times_{T_0\times T_0}T_1$ is the kernel pair of $<Td_0,Td_1>:T_1\rightarrow T_0\times T_0$ in $Set$ (calling $T_{ij}$ the set $T(P_{ij})$ with $1\leq i\leq 4$ and $0\leq j\leq 3$); the others sets in $T$ are the obvious ones, such that the image by $T$ of the vertical precategory diagrams are equivalence relations on the elements in $P_i$ (the rightmost one is trivial, that is, a discrete category with set of objects $P_0$). It is clear then that $Ran_K\circ Set^K(S)\in \check{2\mathbb{P}}$ for every $S\in \check{2\mathbb{P}}$, so that the conditions of Proposition \ref{proposition:stableunitsModelsPresheaves} hold, and hence there is a subreflection having stable units $$\check{H}\vdash \check{I}: \check{2\mathbb{P}}\rightarrow \check{2\mathbb{M}}.$$

$\check{2\mathbb{P}}$ can be identified with $2Cat$, the category of all 2-categories, and $\check{2\mathbb{M}}$ with the category $2Preord$ of all 2-preorders (cf.\ the beginning of $\S 5$ in \cite{X:2ml}).

The reflection $2Cat\rightarrow 2Preord$ was showed to be very-well-behaved in \cite{X:2ml}, where the monotone-light factorization was displayed and showed to be non-trivial ($\mathcal{E}'_{\check{I}}\neq\mathcal{E}_{\check{I}}$). The existence of a monotone-light factorization is a consequence of Theorem \ref{theorem:stable units+enough edm}: for every 2-category $M\in\check{2\mathbb{P}}$, there is an effective descent morphism $p:E\rightarrow M$ in $\hat{2\mathbb{P}}$ which is also an effective descent morphism in $\check{2\mathbb{P}}=2Cat$.

In \cite{X:2ml}, the proof that $\check{H}\vdash\check{I}:2Cat\rightarrow 2Preord$ has stable units is much more complicated than in here. In that paper, one used to the full extent the results in \cite{X:GenConComp} to prove directly that the reflection $\check{H}\vdash\check{I}$ has stable units, while in the present paper this reflection inherits this property (although, to do so, a small theory had to be developed).


\subsection{Discussion of the very-well-behaved reflection from n-categories into n-preorders}\label{subsec:n-categories into n-preorders}

This subsection generalizes the example of the previous subsection, whose results are subsumed here. Anyway, the somewhat repeating of subjects helps the exposition and discussion of these examples, which can be displayed using different theoretical paths.\\

Consider the category $n\mathbb{P}$ generated by all the 2-precategory diagrams below (cf.\ \cite[\S 10]{X:nml} and last subsection \ref{subsec:2-categories into 2-preorders}), for all integers $i$, $j$ and $k$ such that $0\leq i<j<k\leq n$,


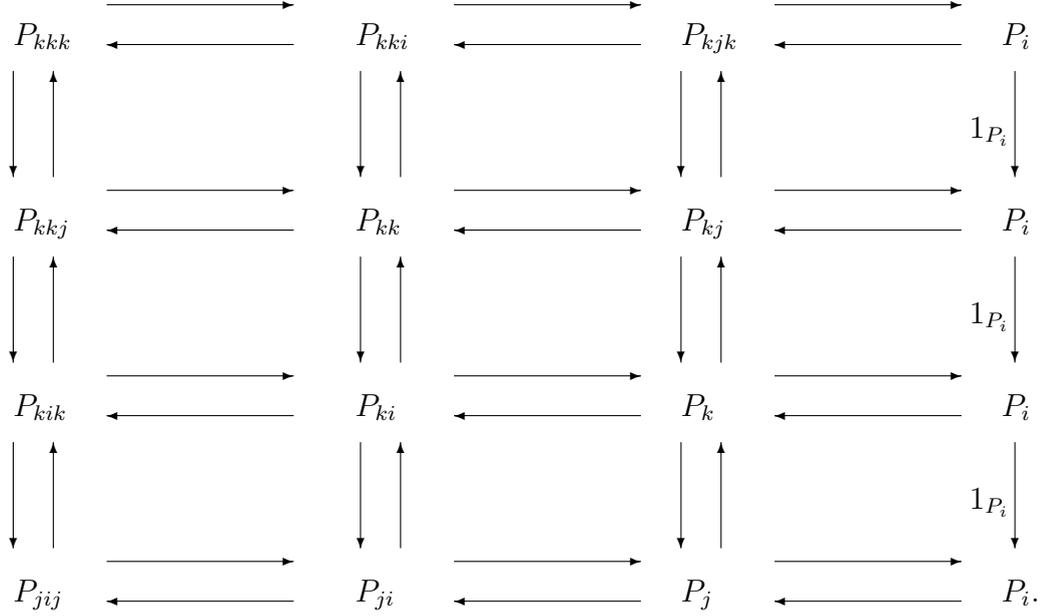
\begin{picture}(400,160)

\put(0,140){$P_{kkk}$}
\put(35,155){\vector(1,0){70}}\put(105,140){\vector(-1,0){70}}

\put(128,140){$P_{kki}$}
\put(165,155){\vector(1,0){70}}\put(235,140){\vector(-1,0){70}}

\put(250,140){$P_{kjk}$}
\put(285,155){\vector(1,0){70}}\put(355,140){\vector(-1,0){70}}
\put(370,140){$P_i$}

\put (0,130){\vector(0,-1){40}}\put (15,90){\vector(0,1){40}}

\put (130,130){\vector(0,-1){40}}\put (145,90){\vector(0,1){40}}

\put (250,130){\vector(0,-1){40}}\put (265,90){\vector(0,1){40}}

\put (358,105){$1_{P_i}$}\put (375,130){\vector(0,-1){40}}

\put(0,70){$P_{kkj}$}
\put(35,85){\vector(1,0){70}}\put(105,70){\vector(-1,0){70}}

\put(128,70){$P_{kk}$}
\put(165,85){\vector(1,0){70}}\put(235,70){\vector(-1,0){70}}

\put(250,70){$P_{kj}$}
\put(285,85){\vector(1,0){70}}\put(355,70){\vector(-1,0){70}}
\put(370,70){$P_i$}

\put (0,60){\vector(0,-1){40}}\put (15,20){\vector(0,1){40}}

\put (130,60){\vector(0,-1){40}}\put (145,20){\vector(0,1){40}}

\put (250,60){\vector(0,-1){40}}\put (265,20){\vector(0,1){40}}

\put (358,35){$1_{P_i}$}\put (375,60){\vector(0,-1){40}}

\put(0,0){$P_{kik}$}
\put(35,15){\vector(1,0){70}}\put(105,0){\vector(-1,0){70}}

\put(128,0){$P_{ki}$}
\put(165,15){\vector(1,0){70}}\put(235,0){\vector(-1,0){70}}

\put(250,0){$P_{k}$}
\put(285,15){\vector(1,0){70}}\put(355,0){\vector(-1,0){70}}
\put(370,0){$P_i$}

\put (0,-10){\vector(0,-1){40}}\put (15,-50){\vector(0,1){40}}

\put (130,-10){\vector(0,-1){40}}\put (145,-50){\vector(0,1){40}}

\put (250,-10){\vector(0,-1){40}}\put (265,-50){\vector(0,1){40}}

\put (358,-35){$1_{P_i}$}\put (375,-10){\vector(0,-1){40}}

\put(0,-70){$P_{jij}$}
\put(35,-55){\vector(1,0){70}}\put(105,-70){\vector(-1,0){70}}

\put(128,-70){$P_{ji}$}
\put(165,-55){\vector(1,0){70}}\put(235,-70){\vector(-1,0){70}}

\put(250,-70){$P_j$}
\put(285,-55){\vector(1,0){70}}\put(355,-70){\vector(-1,0){70}}

\put(370,-70){$P_i$\hspace{10pt.}}

\end{picture}

\vspace{80pt}

The number $\frac{(n-1)n(n+1)}{6}$ of these 2-precategory diagrams can easily be obtained by using the counting principle. Remark that these $\frac{(n-1)n(n+1)}{6}$ diagrams are not independent since they share common objects and arrows (every arrow between two common objects), and that if $n=2$ there is only one 2-precategory diagram.\\

Let $K:(n-1)\mathbb{P}\rightarrow n\mathbb{P}$ be the inclusion functor of $(n-1)\mathbb{P}$ in $n\mathbb{P}$.

The objects of $(n-1)\mathbb{P}$ constitute a cogenerating set for $n\mathbb{P}$, because $P_0$ is a generating object for $(n\mathbb{P})^{op}$. Then, by Theorem \ref{theorem:RightKanCogeneratingVeryWellBehaved}, there is a very-well-behaved reflection $$H\vdash I:\hat{n\mathbb{P}}=Set^{\hat{n\mathbb{P}}}\rightarrow \mathbb{M}=Monos(Set^K).$$

This reflection induces a subreflection having stable units $$\check{H}\vdash \check{I}:\check{n\mathbb{P}}\rightarrow \check{\mathbb{M}},$$ for the sketch whose models are those functors $M$ such that $Md_1\circ Md_0=Md_o\circ Md_2$ and $Md_2\circ Md_0=Md_o\circ Md_3$ are pullback squares for every horizontal and vertical precategory diagram. Indeed, $Ran_K\circ Set^K(M)\in \check{n\mathbb{P}}$ for every $M\in \check{n\mathbb{P}}$. Being $K$ a full inclusion, it is well known that $Ran_K\vdash Set^K$ is a full reflection (cf.\ Corollary 3 in \cite[\S X.3]{SM:cat}), hence $Ran_K$ can be seen as the inclusion of the category of ``indiscrete n-graphs" in $\hat{n\mathbb{P}}$. An ``indiscrete n-graph" $T:n\mathbb{P}\rightarrow Set$ is the presheaf such that $T_n=T_{n-1}\times_{T_{n-2}\times T_{n-2}}T_{n-1}$ is the kernel pair of $<Td_0,Td_1>:T_{n-1}\rightarrow T_{n-2}\times T_{n-2}$ in $Set$. It is clear then that $Ran_K\circ Set^K(S)\in \check{n\mathbb{P}}$ for every $S\in \check{n\mathbb{P}}$, so that the conditions of Proposition \ref{proposition:stableunitsModelsPresheaves} hold, and there is a subreflection having stable units $$\check{H}\vdash \check{I}: \check{n\mathbb{P}}\rightarrow \check{n\mathbb{M}}.$$

$\check{n\mathbb{P}}$ can be identified with $nCat$ (cf.\ \cite[\S 10]{X:nml}), the category of all n-categories, and $\check{n\mathbb{M}}$ with the category $nPreord$ of all n-preorders (cf.\ the beginning of $\S 12$ in \cite{X:nml}).

The reflection $nCat\rightarrow nPreord$ was showed to be very-well-behaved in \cite{X:nml}, where the monotone-light factorization was displayed and showed to be non-trivial ($\mathcal{E}'_{\check{I}}\neq\mathcal{E}_{\check{I}}$), for every $n\geq 2$. The existence of a monotone-light factorization is a consequence of Theorem \ref{theorem:stable units+enough edm}: for every n-category $M\in\check{2\mathbb{P}}$, there is an effective descent morphism $p:E\rightarrow M$ in $\hat{n\mathbb{P}}$ which is also an effective descent morphism in $\check{n\mathbb{P}}=nCat$.

In \cite{X:nml}, the proof that $\check{H}\vdash\check{I}:nCat\rightarrow nPreord$ has stable units was made using a quite different path from the one used here. In that paper, the category of n-categories was considered in the context of $\mathcal{V}$-categories; it is known that, considering $\mathcal{V}=Set$ and then iterating n-categories are obtained (see \cite{Kelly:enriched_cat}). There, beginning with the well-behaved reflection $Cat\rightarrow Preord$ and using the \emph{change of enriching category} (see \cite{CruttwellPhD}), it was possible to show that the derived reflection from n-categories into n-preorders had stable units, while in the present paper this reflection inherits this property from larger reflections in which it is contained.

It is also relevant to state that the path in the context of $\mathcal{V}-categories$, in order to establish the stable units property, could be replaced by the process given in \cite{X:GenConComp} which was used for 2-categories in \cite{X:2ml}.


\section*{Acknowledgement}
This work is supported by CIDMA under the FCT (Portuguese Foundation for Science and Technology) Multi-Annual Financing Program for R\&D Units.

\end{document}